\documentclass{amsart}
\pdfoutput=1
\usepackage{amsmath}
\usepackage{amsthm}
\usepackage{amsfonts}
\usepackage{amssymb}
\usepackage{amsbsy}
\usepackage{amscd}
\usepackage{eucal}
\usepackage[pdftex]{graphicx}
\usepackage{caption}

\newtheorem{theorem}{Theorem}

\newtheorem{lemma}[theorem]{Lemma}

\newtheorem{corollary}{Corollary}

\theoremstyle{definition}

\newtheorem{remark}{Remark}

\newcommand{\beq}{\begin{equation}}
\newcommand{\eeq}{\end{equation}}

\newcommand{\Real}{\mathop{\mathrm{Re}}}
\newcommand{\Imag}{\mathop{\mathrm{Im}}}

\newcommand{\rmd}{\mathrm {d}}
\newcommand{\rmi}{\mathrm {i}}
\newcommand{\rmp}{\mathrm {p}}
\newcommand{\rmt}{\mathrm {t}}
\newcommand{\rmR}{\mathrm {R}}
\newcommand{\rmI}{\mathrm {I}}

\newcommand{\N}{\mathbb{N}}
\newcommand{\R}{\mathbb{R}}
\newcommand{\C}{\mathbb{C}}

\newcommand{\cC}{\mathcal{C}}

\newcommand{\hc}{\widehat{c}}

\newcommand{\hf}{\widehat{f}}

\newcommand{\hm}{\widehat{m}}
\newcommand{\hn}{\widehat{n}}
\newcommand{\hM}{\widehat{M}}

\newcommand{\oc}{\mathfrak{c}}
\newcommand{\fc}{\mathfrak{c}}
\newcommand{\oM}{\mathfrak{M}}
\newcommand{\om}{\mathfrak{m}}

\newcommand{\z}{\zeta}

\newcommand{\sinc}{\mathop\mathrm{sinc}}
\newcommand{\ud}{\frac{1}{2}}
\newcommand{\sds}{\strut\displaystyle}
\def\staccrel#1#2{\mathrel{\mathop{#1}\limits_{#2}}}

\begin{document}

\title[Numerical recovery of poles]{Numerical recovery of location and residue of poles of meromorphic functions}

\author[E. De Micheli]{Enrico De Micheli}
\address{\sl IBF -- Consiglio Nazionale delle Ricerche \\ Via De Marini, 6 - 16149 Genova, Italy}
\email{enrico.demicheli@cnr.it}

\author[G. A. Viano]{Giovanni Alberto Viano}
\address{\sl Facolt\'a di Scienze Matematiche, Fisiche e Naturali -- Universit\`a di Genova \\
Via Dodecaneso, 33 - 16146 Genova, Italy}
\email{giovannialbertoviano@gmail.com}

\begin{abstract}
We present a method able to recover location and residue
of poles of functions meromorphic in a half--plane from samples of the function on the real positive 
semi--axis.
The function is assumed to satisfy appropriate asymptotic conditions including, in particular,
that required by Carlson's theorem. The peculiar features of the present procedure
are: (i) it does not make use of the approximation of meromorphic functions by rational
functions; (ii) it does not use the standard methods of regularization of ill--posed
problems. The data required for the determination of the pole parameters (i.e.,
location and residue) are the approximate values of the meromorphic function
on a finite set of equidistant points on the real positive semi--axis. We show that this method is
numerically stable by proving that the algorithm is convergent
as the number of data points tends to infinity and the noise on the input data goes to zero.
Moreover, we can also evaluate the degree of approximation of the estimates
of pole location and residue which we obtain from the knowledge
of a finite number of noisy samples.
\end{abstract}

\maketitle

\section{Introduction}
\label{se:introduction}

A classical problem of numerical complex analysis consists in recovering location
and residue of poles of meromorphic functions.
The classical approach to this problem is based on the approximation by rational 
functions and, in this framework, the Pad\'e approximants play a particularly significant 
role \cite{Baker,Montessus,Suetin,Walsh}.

In this paper we present a completely different method, whose origin
goes back to a much earlier paper written by one of the authors,
in collaboration with Tullio Regge, in connection with the interpolation problem in the
complex angular momentum plane \cite{Regge}. Work on the method continued in \cite{Albino}.
However, soon, we ran up against difficulties related with the ill--posedness
of the numerical analytic continuation. Now, after more than forty years of experience regarding
the regularization of ill--posed problems \cite{Engl,Tikhonov}, we can reconsider the method 
originated at that time, and present a regularized algorithm which is safe from the pathologies 
of ill--posedness.

First, we consider a function $f(z)$, analytic in the half--plane $\Real z>0$, and
satisfying appropriate asymptotic conditions (detailed below in the article) which,
in particular, include that required by Carlson's theorem \cite{Boas}.
First, we suppose that the data set consists of an infinite number of samples 
of the function $f(z)$, taken on a regular grid on the real positive semi--axis, and
moreover, $f(z)$ is assumed to be known exactly: i.e., the input data are noiseless. 
More precisely, denoting by $\{f_N\}_{N=0}^\infty$ ($f_N \doteq f(N+\frac{1}{2})$; $N\in\N$)
the set of input samples and assuming the series $\sum_{N=0}^\infty f_N$ to be absolutely convergent,
as a first result we find an interpolation formula for $f(x)$ ($x\doteq\Real z$)
along with a relation which allows every sample $f_N$ to be expressed in terms of all the other samples.

The analysis performed for the analytic functions is then generalized to the case of a
meromorphic function $f(z)$ with one first order pole in the half--plane $\Real z>0$.
Still assuming to work with a data set made of an infinite number of noiseless samples of $f(x)$,
and supposing, initially, that position and residue of the pole are known, we obtain
a generalization of the previous formula in which each datum $f_N$ can be reconstructed
from all the other samples and from the pole parameters (i.e., position and residue).
Stated in other words, this formula provides us with a set of consistency relations, 
which mutually constrain the values of all the samples and of the pole parameters.
It is exactly this overall consistency which is exploited in order to construct the algorithm for
recovering pole location and residue from the function samples taken on the real positive semi--axis.

The successive step is to consider as input a more realistic data set
$\{f_N^{(\varepsilon)}\}_{N=0}^{N_0}$ ($\varepsilon>0,N_0<\infty$),
made of a finite number of function samples $f_N^{(\varepsilon)}$ perturbed by noise.
The algorithm for pole recovery, defined previously for the case of an infinite set 
of noiseless input data, can be suitably adapted to this different situation. 
More precisely, we can prove that the limit for $N_0\to+\infty$ and 
$\varepsilon\to 0$ of appropriately defined estimates $z_\rmp^{(\varepsilon,N_0)}(n)$
($n\in\N$ fixed) of the pole location $z_\rmp$ tends to $z_\rmp$ when $n$ tends to infinity, i.e.:
${\sds\lim_{n\to+\infty}}\lim_{\substack{{N_0\rightarrow+\infty}\\{\varepsilon\rightarrow 0}}}
z_\rmp^{(\varepsilon,N_0)}(n)=z_\rmp$.
But, in practice, since $N_0$ is necessarily finite and $\varepsilon$ is non--null,
we are faced with the problem that the two limits in the above formula cannot be
interchanged. This delicate point will be discussed in detail
in Section \ref{se:finite}, where it is shown that a proper estimate $z_\rmp^{(\varepsilon,N_0)}$,
close to the true pole position $z_\rmp$, can be obtained for $\varepsilon$ sufficiently small
and $N_0$ sufficiently large.
In the same section we will also show how to evaluate the degree of approximation
to $z_\rmp$ by $z_\rmp^{(\varepsilon,N_0)}$. Analogous arguments (leading to similar results)
are also developed for the problem of finding a suitable estimate $R_\rmp^{(\varepsilon,N_0)}$
of the residue $R_\rmp$.

The paper is organized as follows.
In Section \ref{se:interpolation:analytic} we derive the interpolation formula
for functions analytic in $\Real z>0$, by taking a data set $\{f_N\}_{N=0}^{\infty}$ 
of noiseless samples. In Section \ref{se:interpolation:meromorphic} we obtain 
the interpolation formula for a meromorphic function in the half--plane
$\Real z>0$, which has one first order pole at $z=z_\rmp$ ($\Real z_\rmp>0$),
continuing to assume an input made of an infinite number of noiseless data.
In Section \ref{se:consistency+algorithm:infinite} the consistency relations
mentioned previously are derived, and the algorithm for recovering 
location and residue of the pole from an infinite set of noiseless function samples is presented.

In Section \ref{se:finite} the algorithm for recovering
the pole parameters with a finite number of noisy input samples is given.
In the same section we show how to evaluate the degree of approximation
of the estimate $z_\rmp^{(\varepsilon,N_0)}$ and $R_\rmp^{(\varepsilon,N_0)}$
to $z_\rmp$ and $R_\rmp$. Section \ref{se:numerical} is devoted
to the numerical examples, which illustrate the various steps of the algorithm.
Finally, some conclusions will be drawn, and possible extensions of the present
method will be outlined. In the Appendix we briefly recall some properties of the
Meixner--Pollaczek polynomials, which are extensively used in the paper.

\section{Interpolation formula for a class of functions analytic in the half--plane $\Real z>0$}
\label{se:interpolation:analytic}

Most of the results we present in this article rely on a celebrated theorem by Carlson, which
states the growth properties that a function of a specified class must enjoy in order to be 
determined by its values on a certain set of points. Some preliminary notions are necessary.
The entire function $f(z)$, $z=re^{\rmi\theta}$, is of exponential type (or, of order 1) if
\beq
\limsup_{r\to\infty}\frac{\log\log M(r)}{\log r} = 1,
\label{0.2}
\eeq
where $M(r)$ denotes the maximum modulus of $f(z)$ for $|z|=r$.
In order to specify the rate of growth of a function of exponential type 
in different directions use can be made of the
Phragm\'en--Lindel\"of indicator function:
\beq
h_f(\theta) \doteq \limsup_{r\to\infty}\frac{\log\left|f(re^{\rmi\theta})\right|}{r}.
\label{0.5}
\eeq
Then, we have the following theorem.

\renewcommand{\thetheorem}{A}

\begin{theorem}[Carlson \cite{Boas}]
If $f(z)$ is regular and of exponential type in the half--plane $\Real z \geqslant 0$ and
$h_f(\frac{\pi}{2})+h_f(-\frac{\pi}{2}) < 2\pi$, then $f(z) \equiv 0$ if $f(N) = 0$ for $N=0,1,2,\ldots$.
\end{theorem}

\setcounter{theorem}{0}
\renewcommand{\thetheorem}{\arabic{theorem}}

For our purposes we shall consider a subset of the functions which fulfill Carlson's theorem, 
that is, the functions satisfying the following bound, which can be named \emph{Carlson's bound}:
\beq
h_f(\theta) \leqslant b \, |\sin\theta|
\qquad \left(b < \pi; -\frac{\pi}{2}\leqslant\theta\leqslant\frac{\pi}{2}\right).
\label{0.6}
\eeq
Evidently, the functions satisfying condition \eqref{0.6} satisfy also the assumptions 
of Carlson's theorem.

\vskip 0.2cm

Let $\N=\{0,1,2,\ldots\}$ denote the set of all natural numbers.
Throughout the paper we shall use the notation
$f_N \doteq f\left(N+\frac{1}{2}\right)$, $N\in\N$, to denote the samples of the
function $f(z)$ at the equidistant interpolation nodes $\left(N+\frac{1}{2}\right)$; 
we shall refer to the set $\{f_N\}_{N=0}^\infty$ as the ``data set''.
We can prove the following interpolation theorem.

\begin{theorem}
\label{the:1new}
Assume that the function $f(z)$ $(z\in\C; z=x+\rmi y; x,y\in\R)$ enjoys the following properties:
\begin{itemize}\addtolength{\itemsep}{0.5\baselineskip}
\item[\rm (i)] $f(z)$ is holomorphic in $\Real z > 0$, continuous at $\Real z = 0$;
\item[\rm (ii)] $f(z)$ satisfies Carlson's bound \eqref{0.6};
\item[\rm (iii)] $f(\rmi y)\in L^2(-\infty,+\infty)$;
\item[\rm (iv)] $\sds\sum_{N=1}^\infty \frac{|f_N|}{N} < \infty$.
\end{itemize}
Then the following equality holds for $x>-\ud$:
\beq
f\left(x+\frac{1}{2}\right)=
\sum_{N=0}^\infty f_N\sinc(x-N)
-\frac{\sin\pi x}{2\pi}\int_{-\infty}^{+\infty}
\frac{\sum_{n=0}^\infty c_n\,\psi_n(y)}{(x+\ud-\rmi y)\cosh\pi y}\,\rmd y,
\label{0.7}
\eeq
where $\sinc(t)\doteq \frac{\sin\pi t}{\pi t}$ for $t\neq 0$ and $\sinc(0)\doteq 1$. 
The coefficients $c_n$ are given by:
\beq
c_n = 2\sqrt{\pi}\sum_{N=0}^\infty \frac{(-1)^N}{N!} f_N\, 
P_n\left[-\rmi\left(N+\frac{1}{2}\right)\right] \qquad (n\in\N),
\label{0.8}
\eeq
and the set of functions $\{\psi_n\}_{n=0}^\infty$ is defined by
\beq
\psi_n (y) \doteq \frac{1}{\sqrt{\pi}}\,\Gamma\left(\ud+\rmi y\right)\,P_n(y) 
\qquad (n\in\N;y\in\R),
\label{0.8.bis}
\eeq
where $\Gamma$ is the Euler Gamma function, and $P_n$ denotes the 
Meixner--Pollaczek polynomials $P_n^{(\alpha)}$ with $\alpha=\ud$.
\end{theorem}

\begin{proof}
In view of conditions (i) and (ii), by Cauchy's integral formula we have, 
for $R\rightarrow\infty$ (see Fig. \ref{fig:1}):
\beq
\frac{1}{2\pi\rmi}\oint_{\cC}\frac{f(z)\Gamma(\frac{1}{2}+z)\Gamma(\frac{1}{2}-z)}{z-z'}\,\rmd z=
f(z')\Gamma\left(\frac{1}{2}+z'\right) \Gamma\left(\frac{1}{2}-z'\right),
\label{0.10}
\eeq
where $\cC$ is the path shown in Fig. \ref{fig:1}, 
and $z'$ belongs to the half--plane $\Real z>0$, but $z' \not\in \cC$.
Next, the integral along the path $\cC$ on the left--hand side (l.h.s.) of \eqref{0.10} 
can be evaluated; noting that the integrand has simple
poles at $z_N = N+\frac{1}{2}$, $N\in\N$, which are brought by $\Gamma(\frac{1}{2}-z)$, we have:
\beq
\begin{split}
& \frac{1}{2\pi\rmi}\oint_{\cC}\frac{f(z)\Gamma(\frac{1}{2}+z)\Gamma(\frac{1}{2}-z)}{z-z'}\,\rmd z \\
& \quad =\sum_{N=0}^\infty f_N\,\frac{(-1)^N}{N+\frac{1}{2}-z'}
-\frac{1}{2\pi}\int_{-\infty}^{+\infty}
\frac{f(\rmi y)\Gamma(\frac{1}{2}+\rmi y) \Gamma(\frac{1}{2}-\rmi y)}{\rmi y-z'}\,\rmd y.
\end{split}
\label{0.11}
\eeq

\begin{figure}[tb] 
\begin{center}
\leavevmode
\includegraphics[width=5cm]{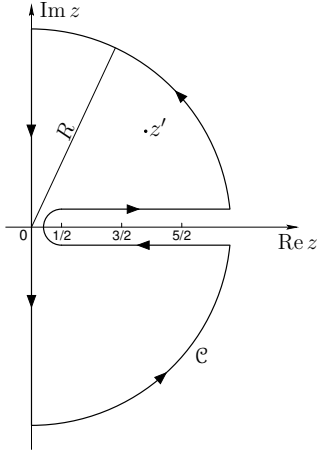}
\caption{\label{fig:1} \small \textsl{Integration path $\cC$ for formula (\protect\ref{0.10}).}}
\end{center}
\end{figure}

\noindent
Then we put $z'=x+\ud+\rmi\epsilon$ ($x>-\frac{1}{2}, x \not\in\N; \epsilon>0$).
Next, deforming appropriately the integration path, we push the point $z'$ up to the 
real axis computing the limit for $\epsilon\rightarrow 0$. By exploiting the relations
$\Gamma(-x)\Gamma(x+1)= -\pi(\sin\pi x)^{-1}$ and
$\Gamma(\frac{1}{2}+\rmi y)\Gamma(\frac{1}{2}-\rmi y)=\pi(\cosh\pi y)^{-1}$,
from \eqref{0.10} and \eqref{0.11} we obtain for $x>-\ud$:
\beq
f\left(x+\frac{1}{2}\right)=
\sum_{N=0}^\infty f_N\sinc(x-N)
-\frac{\sin\pi x}{2\pi}\int_{-\infty}^{+\infty}
\frac{f(\rmi y)}{[(x+\frac{1}{2})-\rmi y]\cosh\pi y}\,\rmd y.
\label{0.12}
\eeq
The integral on the right--hand side (r.h.s.) of \eqref{0.12} converges, 
as can be shown by using assumption (iii) and Schwarz's inequality (see also
the inequalities which will be given at the end of the proof).
Formula \eqref{0.12} has been obtained under the hypothesis $x\not\in\N$; 
however, it is easy to see that the limit for $x\to N$
($N\in\N$) of both sides of \eqref{0.12} leads to the identity 
$f(N+\frac{1}{2})=f(N+\frac{1}{2})$, so that formula
\eqref{0.12} actually holds for every $x>-\frac{1}{2}$.
Let us now introduce the Meixner--Pollaczek polynomials $P^{(1/2)}_n(y)$ 
\cite{Bateman,Pollaczek,Szego}, which are orthonormal with respect to the 
weight function $w(y)$\footnote{Hereafter the superscript (1/2)
in $P^{(1/2)}_n(y)$ will be omitted for simplicity (see also the Appendix).}:
\beq
w(y) \doteq \frac{1}{\pi}\left|\Gamma\left(\frac{1}{2}+\rmi y\right)\right|^2 = \frac{1}{\cosh\pi y}.
\label{0.13}
\eeq
Then we consider the set of functions $\{\psi_n\}_{n=0}^\infty$ defined in \eqref{0.8.bis},
which form an orthonormal basis in $L^2(-\infty,+\infty)$ \cite{Itzykson}.
Now, in view of property (iii), $f(\rmi y)$ may be expanded in the basis $\{\psi_n\}_{n=0}^\infty$:
\beq
f(\rmi y) = \sum_{n=0}^\infty c_n \psi_n(y),
\label{0.15}
\eeq
the convergence being in the $L^2$--norm. By the orthonormal property of the basis 
$\{\psi_n\}_{n=0}^\infty$, the coefficients are given by:
\beq
c_n = \frac{1}{\sqrt{\pi}}\int_{-\infty}^{+\infty} f(\rmi y)
\Gamma\left(\frac{1}{2}-\rmi y\right) P_n(y)\,\rmd y \qquad (n\in\N).
\label{0.16}
\eeq
Next, putting $\rmi y = z$, ($z\in\C$), and evaluating the integral in 
\eqref{0.16} by the complex integration method
along the path $\cC$ shown in Fig. \ref{fig:1}, we obtain formula \eqref{0.8}.
Now, inserting expansion \eqref{0.15} into the integral on the r.h.s. of \eqref{0.12} 
we have the term:
\beq
\int_{-\infty}^{+\infty}\frac{\sum_{n=0}^\infty c_n\,\psi_n(y)}
{(x+\frac{1}{2}-\rmi y)\cosh\pi y}\,\rmd y,
\label{0.17}
\eeq
whose convergence is easily proved by using the Schwarz inequality:
\beq
\begin{split}
& \left(\int_{-\infty}^{+\infty}\!\!\frac{\sum_{n=0}^\infty c_n\,\psi_n(y)}
{(x+\frac{1}{2}-\rmi y)\cosh\pi y}\,\rmd y\right)^2 
\leqslant \sum_{n=0}^\infty\left|c_n\right|^2
\int_{-\infty}^{+\infty}\!\!\frac{1}{|x+\frac{1}{2}-\rmi y|^2\cosh^2\pi y}\,\rmd y \\
& \quad = \left\|f(\rmi y)\right\|^2 \,
\int_{-\infty}^{+\infty}\frac{1}{|x+\frac{1}{2}-\rmi y|^2\cosh^2\pi y}\,\rmd y < \infty.
\end{split}
\label{0.17.bis}
\eeq
Finally, plugging integral \eqref{0.17} in formula \eqref{0.12}, we obtain formula
\eqref{0.7}.
\end{proof}

\begin{remark}
In the numerical analysis (see Section \ref{se:numerical}) sums of the type 
$\sum_{n=0}^\infty c_n\psi_n$ (or similar) are repeatedly used but (obviously) 
truncated at a suitable finite value of $n$, say $n=\hn$. In these cases, sum and integral
in \eqref{0.17} may be interchanged, yielding:
\beq
\ud\int_{-\infty}^{+\infty}\frac{\sum_{n=0}^{\hn}c_n\psi_n(y)}
{[(x+\ud)-\rmi y]\cosh\pi y}\,\rmd y
=\sum_{n=0}^{\hn} c_n\,Q_n\left[-\rmi\left(x+\ud\right)\right],
\label{0.18.new}
\eeq
where the function $Q_n$ is defined by 
\beq
Q_n\left[-\rmi\left(x+\ud\right)\right] \doteq \frac{\rmi}{2\sqrt{\pi}}
\int_{-\infty}^{+\infty}\frac{P_n(y)\,\Gamma(\ud+\rmi y)}{[\rmi(x+\ud)+y]\cosh\pi y}\,\rmd y
\qquad (n\in\N).
\label{0.18.new.1}
\eeq
\label{rem:1}
\end{remark}

\section{Interpolation formula for a function meromorphic in the half--plane $\Real z>0$}
\label{se:interpolation:meromorphic}

Consider now the case of meromorphic functions. For simplicity, we consider
a function $f(z)$, which has only one singularity in the half--plane $\Real z >0$, and
we assume that this singularity is a first order pole, whose residue is $R_\rmp$ ($R_\rmp\neq 0$). 
The extension to the case of several first order poles is straightforward. 
We now prove the following theorem.

\begin{theorem}
\label{the:2new}
Assume that the meromorphic function $f(z)$ has a first order pole at $z=z_\rmp$ with $\Real z_\rmp>0$, 
whose residue is $R_\rmp$ ($R_\rmp\neq 0$). Suppose that $f(z)$ is holomorphic in the half--plane 
$\Real z > \Real z_\rmp$, continuous at $\Real z = 0$, satisfies Carlson's bound \eqref{0.6} 
and conditions (iii) and (iv) of Theorem \ref{the:1new}. Then, the following interpolation 
formula holds:
\beq
\begin{split}
f\left(x+\frac{1}{2}\right) &= \sum_{N=0}^\infty f_N \sinc(x-N)
-\frac{R_\rmp}{\cos(\pi z_\rmp)}\frac{\sin\pi x}{(x+\ud-z_\rmp)} \\
&\quad -\frac{\sin\pi x}{2\pi}\int_{-\infty}^{+\infty}
\frac{\sum_{n=0}^\infty c_n^{(\rmp)}\,\psi_n(y)}{[(x+\ud)-\rmi y]\cosh\pi y}\,\rmd y
\qquad \left(x > -\frac{1}{2}\right),
\label{3.1}
\end{split}
\eeq
where, for $n\in\N$:
\beq 
c_n^{(\rmp)} =
2\sqrt{\pi}\left\{\sum_{N=0}^\infty \frac{(-1)^N}{N!} f_N\, P_n\left[-\rmi\left(N+\frac{1}{2}\right)\right]
-R_\rmp\Gamma\left(\frac{1}{2}-z_\rmp\right) P_n\left(-\rmi z_\rmp\right)\right\}.
\label{3.2}
\eeq
\end{theorem}

\begin{proof}
Proceeding similarly to what we have done in the proof of Theorem \ref{the:1new}, 
and recalling that the function $f(z)$ has a first order pole in $z=z_\rmp$ 
with residue $R_\rmp$, we evaluate the following integral by using the residue theorem:
\beq
\frac{1}{2\pi\rmi}\oint_{\cC}\frac{f(z)\,\Gamma(\frac{1}{2}+z)\,\Gamma(\frac{1}{2}-z)}{z-z'}\,\rmd z
=\pi\left[\frac{f(z')}{\cos\pi z'}+\frac{R_\rmp}{(z_\rmp-z')\cos\pi z_\rmp}\right],
\label{1.44}
\eeq
where $\cC$ is the path shown in Fig. \ref{fig:1} which encloses $z'$ and $z_\rmp$,
but $z',z_\rmp \not\in\cC$. Then, proceeding as in Theorem \ref{the:1new}, we obtain:
\beq
\frac{f(z')}{\cos\pi z'}
=\frac{1}{\pi}\sum_{N=0}^\infty\frac{(-1)^N\,f_N}{N+\frac{1}{2}-z'}
-\frac{R_\rmp}{(z_\rmp-z')\cos\pi z_\rmp}
-\frac{1}{2\pi}\int_{-\infty}^{+\infty}\!\!\!\frac{f(\rmi y)}{(\rmi y-z')\cosh\pi y}\,\rmd y.
\label{1.45}
\eeq
Putting $z'=x+\frac{1}{2}+\rmi\epsilon$, ($x>-\frac{1}{2},x\not\in\N; \epsilon>0$), 
and proceeding as in Theorem \ref{the:1new}, from \eqref{1.45} we obtain:
\beq
\begin{split}
f\left(x+\frac{1}{2}\right)
&=\frac{\sin\pi x}{\pi}\left\{\sum_{N=0}^\infty (-1)^N \frac{f_N}{x-N}
-\frac{\pi R_\rmp}{(x+\frac{1}{2}-z_\rmp)\cos\pi z_\rmp} \right. \\
&\quad\left. -\frac{1}{2}\int_{-\infty}^{+\infty}
\frac{f(\rmi y)}{[(x+\frac{1}{2})-\rmi y]\cosh\pi y}\,\rmd y\right\}.
\label{1.46new}
\end{split}
\eeq
We then expand $f(\rmi y)$ on the orthonormal basis $\{\psi_n\}_{n=0}^\infty$: i.e.,
\beq
f(\rmi y)=\sum_{n=0}^\infty c_n^{(\rmp)}\psi_n(y),
\label{1.47new}
\eeq
where the convergence is in the $L^2$--norm (the superscript $'(\rmp)'$ in $c_n^{(\rmp)}$ is to
recall that these coefficients refer to a function with a pole). In view of the orthonormality property
of the basis $\{\psi_n\}_{n=0}^\infty$ we have:
\beq
c_n^{(\rmp)}=\frac{1}{\sqrt{\pi}}\int_{-\infty}^{+\infty}
f(\rmi y)\,\Gamma\left(\frac{1}{2}-\rmi y\right) P_n(y)\,\rmd y.
\label{1.48}
\eeq
Next, putting $\rmi y=z$, ($z\in\C$), and evaluating this integral by the use of the 
complex integration method along the path $\cC$, we obtain 
formula \eqref{3.2}. Then, plugging \eqref{1.47new} into \eqref{1.46new} we 
obtain the interpolation formula \eqref{3.1}. 
\end{proof}

\section{Consistency relations and the algorithm for pole recovery: case
of input data made of an infinite number of noiseless samples}
\label{se:consistency+algorithm:infinite}

\subsection{Consistency relations}
\label{subse:consistency:infinite}

We now continue to consider a meromorphic function $f(z)$ with only 
one first order pole in $\Real z>0$.
Let us introduce the following function associated with $f(z)$: i.e., 
$h(k;z)\doteq\left[z-\left(k+\frac{1}{2}\right)\right]f(z)$, 
($z\in\C, k\in\N$). We can prove the following theorem.

\begin{theorem}
\label{the:3new}
Assume that the meromorphic function $f(z)$ has a first order pole at $z=z_\rmp$ with 
$\Real z_\rmp>0$, whose residue is $R_\rmp$ ($R_\rmp \neq 0$). 
Suppose that $f(z)$ is holomorphic in the half--plane $\Real z > \Real z_\rmp$,
continuous at $\Real z = 0$, satisfies Carlson's bound \eqref{0.6}, and
\begin{flalign}
\mathrm{(i')}\qquad\sum_{N=0}^\infty |f_N| < \infty. && 
\label{ppp4}
\end{flalign}
Moreover, assume that for every $k\in\N$:
\begin{flalign}
\mathrm{(ii')}\qquad h(k;\rmi y) \in L^2 (-\infty,+\infty) \qquad (k\in\N). &&
\label{ppp44}
\end{flalign}
Then, for every $k\in\N$ the following equalities hold:
\beq
\begin{split}
f_k =& \,(-1)^{k+1}\left\{\sum_{N=0}^\infty (-1)^N f_N \, (1-\delta_{Nk}) \right. \\
&\left. -\ud\int_{-\infty}^{+\infty}
\frac{\sum_{n=0}^\infty c_{n,k}^{(\rmp)}\,\psi_n(y)}{[(x+\ud)-\rmi y]\cosh\pi y}\,\rmd y
-\frac{\pi R_\rmp}{\cos\pi z_\rmp}\right\} \qquad (k\in\N),
\end{split}
\label{4.2}
\eeq
where:
\beq
\begin{split}
c_{n,k}^{(\rmp)} &= 2\sqrt{\pi}
\left\{\sum_{N=0}^\infty \frac{(-1)^N}{N!}(N-k)f_N\,P_n\left[-\rmi\left(N+\frac{1}{2}\right)\right]\right.\\
& \quad\left.-\left(z_\rmp-k-\frac{1}{2}\right)
R_\rmp\,\Gamma\left(\frac{1}{2}-z_\rmp\right)\,P_n\left(-\rmi z_\rmp\right)\right\}\qquad (n,k\in\N),
\end{split}
\label{4.3}
\eeq
the functions $\psi_n(y)$ being defined in \eqref{0.8.bis}.
\end{theorem}

\begin{proof}
In view of the conditions listed above, the results of Theorem \ref{the:2new}
can be applied to the function $h(k;z)$, which has a first order pole in $z=z_\rmp$
with residue $(z_\rmp-k-\ud)R_\rmp$, and whose samples are 
$h(k;N+\ud)=(N-k)f_N$; then, for every $k\in\N$ and 
$x>-\frac{1}{2}$ the following interpolation formula holds (see \eqref{3.1}):
\beq
\begin{split}
& h\left(k;x+\frac{1}{2}\right)= (x-k) f\left(x+\frac{1}{2}\right)
=\sum_{N=0}^\infty (N-k)\,f_N\,\sinc(x-N) \\
&\quad-\frac{\sin\pi x}{2\pi}\int_{-\infty}^{+\infty}\frac{h(k;\rmi y)}{[(x+\ud)-\rmi y]\cosh\pi y}\,\rmd y
-\frac{(z_\rmp-k-\frac{1}{2})\,R_\rmp}{\cos\pi z_\rmp}\frac{\sin\pi x}{x+\frac{1}{2}-z_\rmp}.
\label{4.4}
\end{split}
\eeq
Next, for $k\in\N$ we compute the following derivative:
\beq
\left[\frac{\rmd}{\rmd x}(x-k) f\left(x+\frac{1}{2}\right)\right]_{x=k}
\!\! =\left[f\left(x+\frac{1}{2}\right)+(x-k)f'\left(x+\frac{1}{2}\right)\right]_{x=k} =f_k,
\label{4.5}
\eeq
so that we can equate $f_k$ to the derivative with respect to $x$, computed at $x=k$, 
of the r.h.s. of formula \eqref{4.4}. Therefore, we can formally write:
\beq
\begin{split}
& f_k = \sum_{N=0}^\infty (N-k) f_N \left[\frac{\rmd\sinc(x-N)}{\rmd x}\right]_{x=k} \\
&-\frac{1}{2\pi}\left(\frac{\rmd}{\rmd x}\sin\pi x\right)_{\!\!\!x=k} 
\int_{-\infty}^{+\infty}\frac{h(k;\rmi y)}{[(x+\ud)-\rmi y]\cosh\pi y}\,\rmd y 
-\left(\frac{\sin\pi x}{2\pi}\right)_{x=k} \\
& \times\frac{\rmd}{\rmd x} 
\int_{-\infty}^{+\infty}\frac{h(k;\rmi y)}{[(x+\ud)-\rmi y]\cosh\pi y}\,\rmd y
-\frac{(z_\rmp-k-\frac{1}{2})\,R_\rmp}{\cos\pi z_\rmp}
\left[\frac{\rmd}{\rmd x}\frac{\sin\pi x}{x+\frac{1}{2}-z_\rmp}\right]_{x=k} \hspace{-0.15cm}.
\label{4.6}
\end{split}
\eeq
We have: $\frac{\rmd}{\rmd x}\,\sin\pi x|_{x=k} = \pi (-1)^k$, and
\begin{align}
& \left.\frac{\rmd}{\rmd x}\sinc(x-N)\right|_{x\rightarrow k} = 
\frac{(-1)^{N+k+1}}{N-k}(1-\delta_{Nk}), \label{4.7a} \\
& \left.\frac{\rmd}{\rmd x}\,\,\frac{\sin\pi x}{x+\frac{1}{2}-z_\rmp}\right|_{x=k}
=\frac{\pi \, (-1)^k}{k+\frac{1}{2}-z_\rmp}, \label{4.7b}
\end{align}
which, substituted in \eqref{4.6}, yield formally:
\beq
\begin{split}
& f_k = (-1)^{k+1}\left\{\sum_{N=0}^\infty (-1)^N f_N (1-\delta_{Nk}) \right. \\
& \left.\qquad+\ud \int_{-\infty}^{+\infty}\frac{h(k;\rmi y)}{[(x+\ud)-\rmi y]\cosh\pi y}\,\rmd y 
-\frac{\pi R_\rmp}{\cos\pi z_\rmp} \right\} \qquad (k\in\N).
\end{split}
\label{4.8}
\eeq
It should be observed that the term by term differentiation of the series
$\sum_{N=0}^\infty (N-k) f_N \sinc(x-N)$ is legitimate in view of condition (i$'$). 
We can thus conclude that equalities  \eqref{4.8} are proved.
Next, since $h(k;\rmi y)\in L^2(-\infty,+\infty)$ ($k\in\N, y\in\R$), we can expand
$h(k;\rmi y)$ on the basis $\{\psi_n\}_{n=0}^\infty$:
\beq
h(k;\rmi y) = \sum_{n=0}^\infty c_{n,k}^{(\rmp)} \psi_n(y) \qquad (k\in\N, y\in\R),
\label{4.8.bis}
\eeq
which converges in the $L^2$--norm. Following procedures closely analogous to
those used in Theorems \ref{the:1new} and \ref{the:2new}, the explicit
expression of the coefficients $c_{n,k}^{(\rmp)}$, which is given in \eqref{4.3}, is easily obtained.
Finally, inserting the expansion \eqref{4.8.bis} into the integral on the r.h.s. of
\eqref{4.8}, formula \eqref{4.2} follows.
\end{proof}

For every $k\in\N$, formula \eqref{4.2} gives the value of the sample $f_k$ of the function $f(z)$
in terms of the value of all the other samples $f_N$, with $N \neq k$ (notice 
that also in \eqref{4.3} the contribution of the sample $f_N$ to the coefficient $c_{n,k}^{(\rmp)}$ 
is null for $N=k$). Therefore, equations \eqref{4.2} can be regarded as an (infinite) set 
of consistency relations, which make explicit the mutual constraints among the samples of $f(z)$ 
and the pole parameters $z_\rmp$ and $R_\rmp$. 
However, for this purpose, first we need to extend (in part) the results of the previous theorem,
and consider the expansion of the function $h(k;\rmi y)$ on the orthonormal basis 
$\{\psi_n\}_{n=0}^\infty$ for $k\in\R$ (not only for integral values of $k$). What we need
is stated in the following corollary.

\begin{corollary}
\label{cor:1}
Assume for the meromorphic function $f(z)$ all the conditions of Theorem \ref{the:3new}. 
Assume that the function $h(k;z)\doteq(z-k-\ud)f(z)$ ($z\in\C;k\in\R$) satisfies the following 
condition (which substitutes condition $\mathrm{(ii')}$ of Theorem \ref{the:3new}):
\begin{flalign}
\mathrm{(ii'')}\qquad h(k;\rmi y) \in L^2 (-\infty,+\infty) \quad (k\in\R). &&
\label{ppp45}
\end{flalign}
Then the following equality holds for $x>-\ud$ and for any $k\in\R$:
\beq
\begin{split}
& h\left(k;x+\frac{1}{2}\right)= (x-k) f\left(x+\frac{1}{2}\right)
=\sum_{N=0}^\infty (N-k)\,f_N\,\sinc(x-N) \\
&\quad-\frac{\sin\pi x}{2\pi}\int_{-\infty}^{+\infty}
\frac{\sum_{n=0}^\infty c_n^{(\rmp)}(k)\,\psi_n(y)}{[(x+\ud)-\rmi y]\cosh\pi y}\,\rmd y
-\frac{(z_\rmp-k-\frac{1}{2})\,R_\rmp}{\cos\pi z_\rmp}\frac{\sin\pi x}{x+\frac{1}{2}-z_\rmp},
\label{4.99}
\end{split}
\eeq
where the coefficients $c_n^{(\rmp)}(k)$ are given by:
\beq
\begin{split}
c_n^{(\rmp)}(k) &= 2\sqrt{\pi}
\left\{\sum_{N=0}^\infty \frac{(-1)^N}{N!}(N-k)f_N\,P_n\left[-\rmi\left(N+\frac{1}{2}\right)\right]\right.\\
& \quad\left.-\left(z_\rmp-k-\frac{1}{2}\right)
R_\rmp\,\Gamma\left(\frac{1}{2}-z_\rmp\right)\,P_n\left(-\rmi z_\rmp\right)\right\}\qquad (n\in\N, k\in\R),
\end{split}
\label{4.99.bis}
\eeq
and
\beq
\lim_{n\to\infty}c_n^{(\rmp)}(k)=0 \qquad (k\in\R).
\label{4.99.tris}
\eeq
\end{corollary}

\begin{proof}
Applying the results of Theorem \ref{the:2new} to the function $h(k;z)$ with $k\in\R$,
formulae \eqref{4.99} and \eqref{4.99.bis} follow immediately from the interpolation formula \eqref{3.1}
and from \eqref{3.2}, respectively. For any $k\in\R$ the $c_n^{(\rmp)}(k)$ represent the coefficients 
of the expansion of $h(k;\rmi y)$ in terms of the basis $\{\psi_n\}_{n=0}^\infty$, i.e.:
\beq
h(k;\rmi y) = \sum_{n=0}^\infty c_{n}^{(\rmp)}(k) \psi_n(y) \qquad (k\in\R),
\label{4.88.bis}
\eeq
the convergence being in the sense of the $L^2$--norm. 
Finally, from expansion \eqref{4.88.bis} we have for $k\in\R$:
$\left\|h(k;\rmi y)\right\|^2 = \sum_{n=0}^\infty \left|c_n^{(\rmp)}(k)\right|^2$,
which implies \eqref{4.99.tris}.
\end{proof}

\subsection{The algorithm for recovering pole location and residue}
\label{subse:algorithm:infinite}

Let us continue to consider a meromorphic function $f(z)$ with one first order
pole in $z=z_\rmp$ with $\Real z_\rmp>0$, whose residue is $R_\rmp\neq 0$.
Moreover, the conditions required by Corollary \ref{cor:1} are assumed to be 
satisfied by $f(z)$ and its associated function $h(k;z)$ ($k\in\R$).
Now, it is convenient to rewrite the coefficients $c_n^{(\rmp)}(k)$, 
given in \eqref{4.99.bis}, as follows:
\begin{align}
& c_n^{(\rmp)}(k) = \oc_n(k)-\left(\zeta_\rmp - k\right)\tau_n \qquad (n\in\N,k\in\R), \label{4.9} \\
\intertext{where}
& \oc_n(k) \doteq 2\sqrt{\pi}\sum_{N=0}^\infty \frac{(-1)^N}{N!}\,(N-k)\,f_N\,
P_n\!\left[-\rmi\left(N+\frac{1}{2}\right)\right], 
\label{4.10a} \\
& \tau_n \doteq 2\sqrt{\pi} \, R_\rmp \, \Gamma\left(\frac{1}{2}-z_\rmp\right)
\,P_n\left(-\rmi z_\rmp\right), 
\label{4.10b} \\
& \zeta_\rmp \doteq z_\rmp-\frac{1}{2}. \label{4.10c}
\end{align}
Note that for every $k\in\R$ the coefficients $\oc_n(k)$ can be computed from the input 
data set $\{f_N\}_{N=0}^\infty$, and, consequently, can be regarded as known, 
whereas the explicit dependence of the coefficients $c_n^{(\rmp)}(k)$ on the unknown
pole is contained only in the second term on the r.h.s. of \eqref{4.9}. \\
Now, Eq. \eqref{4.99.tris} allows us to connect the unknowns $\z_\rmp$ and $R_\rmp$ to
the input data $\{f_N\}$ through the function $\oc_n(k)$. In fact, 
from \eqref{4.99.tris} and \eqref{4.9} we have:
\beq
\lim_{n\to+\infty}\oc_{n}(k) = \lim_{n\to+\infty}\left(-\tau_n\,k+\zeta_\rmp\tau_n\right)
\qquad (k\in\R),
\label{4.131}
\eeq
which shows that, in the limit for $n$ tending to infinity, the coefficients of the form 
$\oc_{n}(k)$, which is linear in $k$, are related to the unknown pole parameters. 
Now, in order to obtain $\oc_{n}(k)$ ($n\in\N,k\in\R$) from the input data, 
it is sufficient to compute the coefficients $\fc_{n,k}$ for any two integer values of $k$,
say $k_1$ and $k_2$, and successively for every $n\in\N$ interpolate linearly $\fc_{n,k_1}$ and
$\fc_{n,k_2}$ to yield
\beq
\oc_{n}(k) = m_n k + q_n \quad\quad (n\in\N, k\in\R).
\label{4.132}
\eeq
In this way, for any $n\in\N$ we can associate the coefficients $m_n$ and $q_n$ with the function 
samples $\{f_N\}$, i.e., for any $n\in\N$:
\beq
\label{4.132bis}
\left\{f_N\right\}_{N=0}^\infty \xrightarrow[\mathrm{Formula}~\eqref{4.10a}]{}
\left\{\fc_{n,k_1},\fc_{n,k_2}\right\} \xrightarrow[\mathrm{~Linear~interpolation~in~} k~]{} 
\left(m_n,q_n\right).
\eeq
It should be recalled that in the current case we are assuming to know 
an infinite number of noiseless input samples $\{f_N\}_{N=0}^\infty$,
which amounts to saying that the calculated coefficients $\fc_{n,k}$ are exact.
As will be discussed in the next section, in practice, when only a finite number of noisy function 
samples is available and, consequently, only an approximation of the coefficients 
$\fc_{n,k}$ is computable, the scheme in \eqref{4.132bis} needs to be generalized. \\
Comparing \eqref{4.131} and \eqref{4.132}, it can be seen that, for finite values of $n$, 
the computed coefficients $m_n$ and $q_n$ can be considered estimates of $(-\tau_n)$ and 
$(\z_\rmp \tau_n)$, respectively (i.e., for $n \gg 1$, $m_n \sim -\tau_n$ and 
$q_n \sim \z_\rmp \tau_n$), which Eq. \eqref{4.131} guarantees to be such that:
\begin{subequations}
\label{4.133}
\begin{align}
& \lim_{n\to+\infty} m_n = -\lim_{n\to+\infty} \tau_n, \label{4.133.a} \\[+4pt]
& \lim_{n\to+\infty} q_n = \z_\rmp \lim_{n\to+\infty} \tau_n. \label{4.133.b}
\end{align}
\end{subequations}
Now, Eqs. \eqref{4.133} guide us to define, for every $n\in\N$, 
the approximate pole position $\z_\rmp(n)$ as
\beq
\z_\rmp (n) \doteq -\frac{q_n}{m_n} \qquad (n\in\N),
\label{4.134}
\eeq 
(in order to avoid proliferation of symbols, we denote the approximate pole position computed 
at a certain value of $n$ by $\zeta_\rmp(n)$, making explicit the dependence on $n$; instead, the 
true pole position is simply denoted by $\zeta_\rmp$. Moreover, for simplicity, we will refer
interchangeably to $z_\rmp$ and $\z_\rmp$ as the pole position).
Finally, Eqs. \eqref{4.133} and \eqref{4.134} guarantee that
\beq
\lim_{n\to+\infty}\z_\rmp (n) = \z_\rmp,
\label{4.135}
\eeq
which, explicitly, reads:
\begin{subequations}
\label{4.136} 
\begin{align}
& \Real\z_\rmp = -\lim_{n\to+\infty}
\frac{\Real q_n \Real m_n + \Imag q_n \Imag m_n}{|m_n|^2}, \label{4.136a} \\[+6pt]
& \Imag\z_\rmp = -\lim_{n\to+\infty}
\frac{\Imag q_n \Real m_n - \Real q_n \Imag m_n}{|m_n|^2}. \label{4.136b}
\end{align}
\end{subequations}
Once $\zeta_\rmp$ has been recovered (and, accordingly, also $z_\rmp$ by formula
\eqref{4.10c}), also the residue can be readily recovered from the data.
In fact, for every $n\in\N$ we can define the approximate residue $R_\rmp(n)$ as
\beq
\label{4.137}
R_\rmp(n) \doteq -\frac{m_n}{2\sqrt{\pi}\,\Gamma\left(\frac{1}{2}-z_\rmp\right)
\,P_n\left(-\rmi z_\rmp\right)} \qquad (n\in\N).
\eeq
Finally, Eqs. \eqref{4.10b}, \eqref{4.133.a} and \eqref{4.137} allow us to state
\beq
\lim_{n\to+\infty} R_\rmp(n) = -\frac{1}{2\sqrt{\pi}\,\Gamma(\ud-z_\rmp)}\lim_{n\to+\infty}
\frac{m_n}{P_n(-\rmi z_\rmp)} = R_\rmp.
\label{4.138}
\eeq

\section{Consistency relations and the algorithm for pole recovery: case
of input data made of a finite number of noisy samples}
\label{se:finite}

In practice, actual data handling requires the analysis of more
realistic situations in which the input data set is made of a
finite number of noisy data: the data set now is
$\{f_N^{(\varepsilon)}\}_{N=0}^{N_0}$, where $\varepsilon$
characterizes a bound on the noise that will be specified below.
Various models of noise are actually possible.
Since in our case the data $f_N$ are required to vanish as $N\rightarrow +\infty$,
we assume a noise model such that the relative error remains bounded, namely, we write:
$f_N^{(\varepsilon)}= (1+\nu_N^{(\varepsilon)})f_N$, where $\nu_N^{(\varepsilon)}$
denotes a noise term such that $\left|\nu_N^{(\varepsilon)}\right|\leqslant\varepsilon$.
It follows that:
$\left|(f_N^{(\varepsilon)}-f_N)/f_N\right|\leqslant\varepsilon$,
$f_N\neq 0$, $\varepsilon > 0$ constant; evidently, if
$f_N=0$ the relative error becomes meaningless, so in
this particular case we simply assume that
$\lim_{\varepsilon\rightarrow 0} \left|f_N^{(\varepsilon)}\right|=0$.

\subsection{Algorithm for recovering pole location and residue}
\label{subse:algorithm:finite}

When $\varepsilon>0$ and $N_0<\infty$, the coefficients
$\oc_{n}(k)$ in formula \eqref{4.10a} can be computed
only approximately. Then, for fixed values of $\varepsilon$ and
$N_0$, we can define for any $n\in\N$ and $k\in\R$ the following
approximate coefficients:
\beq
\oc_{n}^{\,(\varepsilon,N_0)}(k) \doteq 
2\sqrt{\pi} \sum_{N=0}^{N_0} \frac{(-1)^N}{N!} (N-k)
f_N^{(\varepsilon)} \, P_n\left[-\rmi\left(N+\frac{1}{2}\right)\right]
\quad (\varepsilon>0, N_0<\infty).
\label{4.10abis}
\eeq
Evidently, $\oc_{n}^{\,(0,\infty)}(k)\equiv\oc_{n}(k)$.
We can prove the following lemma.

\begin{lemma}
\label{lem:1}
For every $n\in\N$ and $k\in\R$, the following statement holds:
\beq
\lim_{\substack{{N_0\rightarrow+\infty}\\{\varepsilon\rightarrow 0}}}
\oc_{n}^{\,(\varepsilon,N_0)}(k) = \oc_{n}^{\,(0,\infty)}(k)=\oc_{n}(k).
\label{D5.1}
\eeq
\end{lemma}

\begin{proof}
Consider
\beq
\begin{split}
& \frac{\oc_{n}^{\,(0,\infty)}(k)-\oc_{n}^{\,(\varepsilon,N_0)}(k)}{2\sqrt{\pi}}=
\left\{ \sum_{N=0}^{N_0}\frac{(-1)^N}{N!} (N-k)
\left[f_N-f_N^{(\varepsilon)}\right] \, P_n\left[-\rmi\left(N+\frac{1}{2}\right)\right] \right. \\
&\qquad
\left.+\sum_{N=N_0+1}^\infty\frac{(-1)^N}{N!} (N-k) f_N \, 
P_n\left[-\rmi\left(N+\frac{1}{2}\right)\right]\right\}.
\label{D5.2}
\end{split}
\eeq
We know that the series
$2\sqrt{\pi}\sum_{N=0}^\infty\frac{(-1)^N}{N!}(N-k)f_N P_n\left[-\rmi\left(N+\frac{1}{2}\right)\right]$
converges to $\oc_{n}^{\,(0,\infty)}(k)$, which is finite
for every finite $n\in\N$ and $k\in\R$.
The latter statement follows from formula \eqref{4.9}: in fact,
$\left|c_{n}^{\,(\rmp)}(k)\right|<\infty$ since they are the coefficients of
the expansion of $h(k;\rmi y)$, and
$|(z_\rmp-k-\frac{1}{2}) R_\rmp \Gamma(\frac{1}{2}-z_\rmp)P_n(-\rmi z_\rmp)|<\infty$ for $n\in\N$, 
$k\in\R$ (of course, $z_\rmp \neq N+\frac{1}{2}$, which merely means that the
pole cannot be located on the input datum). It follows that the
second sum on the r.h.s. of \eqref{D5.2} vanishes as
$N_0\to+\infty$. Concerning the first term, we may write the inequality:
\beq
\begin{split}
& \left| \sum_{N=0}^{N_0}\frac{(-1)^N}{N!} (N-k)
\left[f_N-f_N^{(\varepsilon)}\right] \,
P_n\left[-\rmi\left(N+\frac{1}{2}\right)\right] \right| \\
&\qquad
\leqslant\varepsilon\sum_{N=0}^{N_0}\frac{|N-k|}{N!}\left|f_N\right|\,
\left|P_n\left[-\rmi\left(N+\frac{1}{2}\right)\right]\right|,
\end{split}
\label{D5.3}
\eeq
where the assumption made on the noise has been
used. Next, by rewriting the Pollaczek polynomials
$P_n\left[-\rmi\left(N+\frac{1}{2}\right)\right]$ as
\beq
P_n\left[-\rmi\left(N+\frac{1}{2}\right)\right]= \sum_{j=0}^n
p_j^{(n)}\left(N+\frac{1}{2}\right)^j,
\label{D5.4}
\eeq
and
substituting this expression in the r.h.s. of inequality \eqref{D5.3}, we obtain
\beq
\varepsilon\sum_{N=0}^{N_0}\frac{|N-k|}{N!}\left|f_N\right|
\left[\sum_{j=0}^n \left|p_j^{(n)}\right| \,
\left(N+\frac{1}{2}\right)^j\right].
\label{D5.5}
\eeq
Next, we compute the limit for $N_0\to+\infty$.
Since the sum $\sum_{j=0}^n p_j^{(n)}\left(N+\frac{1}{2}\right)^j$ is finite, 
the order of the sums in \eqref{D5.5} may be exchanged:
\beq
\varepsilon\sum_{j=0}^n \left|p_j^{(n)}\right| \sum_{N=0}^{\infty}
\frac{|N-k|}{N!}\left|f_N\right|
\left(N+\frac{1}{2}\right)^j.
\label{D5.6}
\eeq
The inner series
$\sum_{N=0}^\infty\frac{|N-k|}{N!}|f_N|(N+\frac{1}{2})^j$
is evidently convergent in view of assumption ($\mathrm{i'}$) of Theorem \ref{the:3new}, and
therefore, the expression in \eqref{D5.6} vanishes for $\varepsilon\to 0$.
Statement \eqref{D5.2} is thus proved. 
\end{proof}

Let us now tackle the problem of recovering, in practice, the position of the pole.
For this purpose we follow a procedure analogous to that described in 
Subsection \ref{subse:algorithm:infinite},
using now the computable coefficients $\oc_{n}^{\,(\varepsilon,N_0)}(k)$ 
instead of the (exact but unknown) coefficients $\oc_{n}(k)$. 
Then, for given fixed values of $\varepsilon>0$ and $N_0<\infty$, 
the actual implementation is realized by the following procedure:

\begin{itemize}
\item[\bf 1.] For every $n\in\N$ compute by means of formula \eqref{4.10abis}
the coefficients $\oc_{n}^{\,(\varepsilon,N_0)}(k)$ for some integral values of $k$, say $k=0,\ldots,k_*$.
\item[\bf 2.] Since $\oc_{n}^{\,(\varepsilon,N_0)}(k)$ is a linear function of $k$ (see \eqref{4.10abis}), 
we associate by a linear regression procedure (in $k$) the set of computed ``noisy'' coefficients 
$\{\oc_{n,k}^{\,(\varepsilon,N_0)}\}_{k=0}^{k_*}$ with the linear form
\beq
\oc_{n}^{\,(\varepsilon,N_0)}(k) = m_n^{\,(\varepsilon,N_0)}\,k + q_n^{\,(\varepsilon,N_0)}
\qquad (n\in\N,k\in\R).
\label{Q5.1}
\eeq
For every $n\in\N$, we therefore link the coefficients $m_n^{\,(\varepsilon,N_0)}$ 
and $q_n^{\,(\varepsilon,N_0)}$ to the noisy input data $\{f_N^{\,(\varepsilon)}\}$
according to the scheme (see also \eqref{4.132bis}):
\beq
\label{Q5.1bis}
\left\{f_N^{\,(\varepsilon)}\right\}_{\!N=0}^{\!N_0} \xrightarrow[\mathrm{Formula} ~\eqref{4.10abis}]{}
\left\{\fc_{n,k}^{\,(\varepsilon,N_0)}\right\}_{\!k=0}^{\!k_*}
\xrightarrow[\mathrm{~Linear~regression}~]{}
\left(m_n^{\,(\varepsilon,N_0)},q_n^{\,(\varepsilon,N_0)}\right).
\nonumber
\eeq
\item[\bf 3.] For every $n\in\N$, compute the function
$\z_\rmp^{\,(\varepsilon,N_0)}(n)$ as (see also \eqref{4.134})
\beq
\z_\rmp^{\,(\varepsilon,N_0)}(n) \doteq -\frac{q_n^{\,(\varepsilon,N_0)}}{m_n^{\,(\varepsilon,N_0)}}.
\label{Q5.2}
\eeq
Now, Lemma \ref{lem:1} informs us that:
\beq
\oc_n^{\,(\varepsilon,N_0)}(k)
\xrightarrow[\substack{{N_0\rightarrow+\infty}\\{\varepsilon\rightarrow 0}}] {}   
\oc_n^{\,(0,\infty)}(k) \equiv \oc_n(k)
\qquad (n\in\N,k\in\R),
\label{Q5.6}
\eeq
and, consequently, we have for every $n\in\N$ (see also \eqref{4.132}):
\begin{subequations}
\label{Q5.7}
\begin{align}
& m_n^{\,(\varepsilon,N_0)}
\xrightarrow[\substack{{N_0\rightarrow+\infty}\\{\varepsilon\rightarrow 0}}] {}   
m_n^{\,(0,\infty)} \equiv m_n, \label{Q5.7a} \\[+5pt]
& q_n^{\,(\varepsilon,N_0)}
\xrightarrow[\substack{{N_0\rightarrow+\infty}\\{\varepsilon\rightarrow 0}}] {}   
q_n^{\,(0,\infty)} \equiv q_n. \label{Q5.7b}
\end{align}
\end{subequations}
Accordingly, from formulae \eqref{4.134} and \eqref{Q5.2} it follows
\beq
\lim_{\substack{{N_0\rightarrow+\infty}\\{\varepsilon\rightarrow 0}}}
\z_\rmp^{\,(\varepsilon,N_0)}(n) = \z_\rmp^{\,(0,\infty)}(n) = \z_\rmp(n) \qquad (n\in\N).
\label{L1}
\eeq
\item[\bf 4.] Finally, in view of formula \eqref{4.135}, 
we obtain the formula for recovering the position of the pole:
\beq
\lim_{n\to+\infty}\left(
\lim_{\substack{{N_0\rightarrow+\infty}\\{\varepsilon\rightarrow 0}}}
\z_\rmp^{\,(\varepsilon,N_0)}(n)\right) = \z_\rmp,
\label{L2}
\eeq
or, equivalently, defining $z_\rmp^{\,(\varepsilon,N_0)}(n)\doteq\z_\rmp^{\,(\varepsilon,N_0)}(n)+\ud$,
and in view of \eqref{4.10c}:
\beq
\lim_{n\to+\infty}\left(
\lim_{\substack{{N_0\rightarrow+\infty}\\{\varepsilon\rightarrow 0}}}
z_\rmp^{\,(\varepsilon,N_0)}(n)\right) = z_\rmp.
\label{L2.bis}
\eeq
\end{itemize}

\vskip 0.25 cm

For its actual implementation, formula \eqref{L2.bis} deserves a deeper analysis.
To begin with, assume (unrealistically) that we can perform the inner limit for $N_0\to+\infty$ and
$\varepsilon\to 0$ to get the function $z_\rmp^{\,(0,\infty)}(n)$.
Now, the outer limit in \eqref{L2.bis}, which is a direct consequence of limit \eqref{4.99.tris},
tells us that $z_\rmp^{\,(0,\infty)}(n)$ is expected to become close to $z_\rmp$ 
from a certain value of $n$ on (say, $n>n_\mathrm{min}$), in correspondence of the values
of $n$ for which $c_n^{(\rmp)}(k)$ becomes nearly zero (see \eqref{4.99.tris}).
This means that, in practice, in the plot of $z_\rmp^{\,(0,\infty)}(n)$ against $n$
we should be able to identify a ``\emph{range of convergence}'', that is, a set of $n$--values 
where $z_\rmp^{\,(0,\infty)}(n)$ is nearly constant (actually, since in general $z_\rmp\in\C$,
two ``\emph{ranges of convergence}'', one for the real and one for the imaginary part, separately).
More precisely, for an arbitrary constant $\eta>0$, we expect to find an integer 
$n_\mathrm{min}=n_\mathrm{min}(\varepsilon,N_0;\eta)$ such that:
\beq
\left|z_\rmp^{(0,\infty)}(n)-z_\rmp\right|< \eta
\qquad \mathrm{for}~~n\geqslant n_\mathrm{min}(0,\infty;\eta).
\label{pullo.0}
\eeq
Notice that, in this case with $\varepsilon=0$ and $N_0=\infty$ and in view of \eqref{4.99.tris}, 
the \emph{range of convergence} is expected to be superiorly unlimited.

\vskip 0.3 cm

Now, in a realistic situation $\varepsilon$ cannot be null,
$N_0$ is necessarily finite, and both must be regarded as fixed.
Therefore the inner limit in \eqref{L2.bis} cannot be actually performed.
This fact has consequences on the algorithm in view of the fact that
the two limits in \eqref{L2.bis} cannot be interchanged.
In order to see this, let us define, in close analogy with formula \eqref{4.9}
(see also \eqref{4.99.bis}), the following approximate coefficients:
\beq
c_{n}^{\,(\rmp;\varepsilon,N_0)}(k)
\doteq\oc_{n}^{\,(\varepsilon,N_0)}(k)-\left(\zeta_\rmp-k\right)\tau_n  \qquad 
(n\in\N,k\in\R;\varepsilon>0, N_0<\infty),
\label{pullo.1}
\eeq
where $\oc_{n}^{\,(\varepsilon,N_0)}(k)$, $\tau_n$ and $\z_\rmp$ 
are given by \eqref{4.10abis}, \eqref{4.10b} and \eqref{4.10c}, respectively.
Comparing \eqref{pullo.1} with \eqref{4.99.bis},
and by Lemma \ref{lem:1}, it follows: $c_{n}^{(\rmp;0,\infty)}(k)=c_{n}^{(\rmp)}(k)$. 
Now, we have:
\beq
\lim_{n\to+\infty}\left(
\lim_{\substack{{N_0\rightarrow+\infty}\\{\varepsilon\rightarrow 0}}} 
c_{n}^{\,(\rmp;\varepsilon,N_0)}(k)
\right) 
\neq
\lim_{\substack{{N_0\rightarrow+\infty}\\{\varepsilon\rightarrow 0}}}
\left(\lim_{n\to+\infty}c_{n}^{\,(\rmp;\varepsilon,N_0)}(k)\right).
\label{pullo.2}
\eeq
In fact, the l.h.s. of \eqref{pullo.2} is null since $c_{n}^{(\rmp;0,\infty)}(k)$ are
the coefficients of expansion \eqref{4.88.bis}. Instead, for what concerns the r.h.s. 
of \eqref{pullo.2} we have, by using the asymptotic formulae \eqref{B3} and \eqref{B33} 
for the Pollaczek polynomials, with $N_0<\infty$, $\varepsilon>0$:
\beq
\begin{split}
& \left|c_{n}^{\,(\rmp;\varepsilon,N_0)}(k)\right| \\
&\staccrel{\sds\sim}{n\gg 1}
2\sqrt{\pi}\left|(-1)^{N_0}\frac{(N_0-k)\,f_{N_0}^{(\varepsilon)}}{(N_0!)^2}\,(2n)^{N_0}
-(\z_\rmp-k)R_\rmp\frac{\Gamma(\ud-z_\rmp)}{\Gamma(\ud+z_\rmp)}\,(2n)^{z_\rmp-1/2}\right|,
\end{split}
\label{pullo.3}
\eeq
which tends to infinity as $n\to+\infty$. Now, since Eq. \eqref{L2.bis} is a direct 
consequence of the fact that the l.h.s. of \eqref{pullo.2} is null, 
then formula \eqref{pullo.2} does not allow the limits in \eqref{L2.bis} to be switched.

Assume now (more realistically) that $\varepsilon$ and $N_0$ take on the fixed values
$\overline{\varepsilon}$ and $\overline{N}_0$, respectively: i.e.,
$\varepsilon\equiv\overline{\varepsilon}$ and $N_0\equiv\overline{N}_0$. 
In view of \eqref{L2.bis} we have therefore to deal with the following limit:
$\lim_{n\to+\infty}z_\rmp^{\,(\overline{\varepsilon},\overline{N}_0)}(n)$.

Immediate consequence of the divergence of 
$c_{n}^{(\rmp;\overline{\varepsilon},\overline{N}_0)}(k)$
as $n\to+\infty$ (with $\overline{\varepsilon}>0$ and $\overline{N}_0<\infty$) 
is the divergence of $z_\rmp^{(\overline{\varepsilon},\overline{N}_0)}(n)$ from 
$z_\rmp$ as $n\to+\infty$ (see \eqref{L2.bis}).
Therefore, in the actual analysis of $z_\rmp^{(\overline{\varepsilon},\overline{N}_0)}(n)$,
$n$ cannot be pushed to infinity, but must be stopped before this divergence sets in. 
However, if $\overline{\varepsilon}$ is ``sufficiently small'' and 
$\overline{N}_0$ is ``sufficiently large'', 
then, according to formula \eqref{L1}, $z_\rmp^{\,(\overline{\varepsilon},\overline{N}_0)}(n)$ 
(at fixed $n$) is expected to be close to $z_\rmp^{\,(0,\infty)}(n)$, and consequently, for not 
too large values of $n$, say $n<n_\mathrm{max}$ (and with $n>n_\mathrm{min}$), 
we will have also $z_\rmp^{\,(\overline{\varepsilon},\overline{N}_0)}(n)\simeq z_\rmp$.
Therefore, in the plot of $z_\rmp^{\,(\overline{\varepsilon},\overline{N}_0)}(n)$
against $n$ we aim at identifying a range of $n$--values (the
``\emph{range of convergence}''), now limited superiorly, where 
$z_\rmp^{\,(\overline{\varepsilon},\overline{N}_0)}(n)$ is nearly constant.
More precisely, given an arbitrary constant $\eta>0$ 
(whose value determines the allowed range of variability of the estimate),
our goal is to find two integers 
$n_\mathrm{min}(\varepsilon,N_0;\eta)$ and $n_\mathrm{max}(\varepsilon,N_0;\eta)$
and a value $z^{(\overline{\varepsilon},\overline{N}_0)}_\rmp$, 
which represents the estimate of the pole position at the given values 
$\varepsilon=\overline{\varepsilon}$ and $N_0=\overline{N}_0$, such that:
\beq
\left|z_\rmp^{(\overline{\varepsilon},\overline{N}_0)}(n)-
z^{(\overline{\varepsilon},\overline{N}_0)}_\rmp\right|< \eta
\qquad \mathrm{for}~~n_\mathrm{min}(\varepsilon,N_0;\eta)\leqslant n 
\leqslant n_\mathrm{max}(\varepsilon,N_0;\eta).
\label{pullo.4}
\eeq
Since $z_\rmp^{(\overline{\varepsilon},\overline{N}_0)}(n)$ may vary 
significantly within the $2\eta$--\emph{wide interval} defined in \eqref{pullo.4}, 
in the actual numerical implementation of the 
algorithm (see Section \ref{se:numerical}), once the range $[n_\mathrm{min},n_\mathrm{max}]$ 
has been detected (if any), it can be taken as estimate $z^{(\overline{\varepsilon},\overline{N}_0)}_\rmp$ 
of the pole position the sample mean of the values of 
$z_\rmp^{\,(\overline{\varepsilon},\overline{N}_0)}(n)$ within this range, 
while the sample standard deviation can be used as an estimate of the uncertainty
\footnote{For simplicity, we cease henceforth to use the notation $\overline{\varepsilon},\overline{N}_0$
that we adopted in this subsection to emphasize the case when ${\varepsilon},{N}_0$ take on fixed values.}.
Finally, in view of the arguments discussed above (and comparing \eqref{pullo.4} with \eqref{pullo.0})
it is worth observing that $\lim_{\substack{{N_0\rightarrow+\infty}\\{\varepsilon\rightarrow 0}}}
n_\mathrm{max}(\varepsilon,N_0;\eta) = +\infty$.

\vskip 0.3 cm

We can now move on to consider the problem of evaluating the residue $R_\rmp$. 
Inspired by \eqref{4.137} and \eqref{Q5.2} (and recalling that
$z_\rmp^{\,(\varepsilon,N_0)}(n)=\z_\rmp^{\,(\varepsilon,N_0)}(n)+\ud$),
we compute, for every $n\in\N$, the function
\beq
R_\rmp^{\,(\varepsilon,N_0)}(n) \doteq 
-\frac{m_n^{\,(\varepsilon,N_0)}}{2\sqrt{\pi}\,
\Gamma\left(\frac{1}{2}-z_\rmp^{\,(\varepsilon,N_0)}\right)
\,P_n\left(-\rmi z_\rmp^{\,(\varepsilon,N_0)}\right)} \qquad (n\in\N).
\label{L4}
\eeq 
Then, from \eqref{4.138}, \eqref{Q5.7a}, and \eqref{L2.bis} we have:
\beq
\lim_{n\to+\infty}\left(
\lim_{\substack{{N_0\rightarrow+\infty}\\{\varepsilon\rightarrow 0}}}
R_\rmp^{\,(\varepsilon,N_0)}(n)\right) = R_\rmp.
\label{L5}
\eeq
The structure of Eq. \eqref{L5} is equal to that of Eq. \eqref{L2.bis}.
Then, the arguments used earlier for estimating the pole position by the
analysis of Eq. \eqref{L2.bis} can be used similarly for estimating the residue 
by means of Eq. \eqref{L5}.
Therefore, if $\varepsilon$ is ``sufficiently small'' and $N_0$ is ``sufficiently large''
$R_\rmp^{\,(\varepsilon,N_0)}(n)$ is expected to show, as a function of $n$,
a ``\emph{range of convergence}'' within which $R_\rmp^{\,(\varepsilon,N_0)}(n)$
is nearly constant. The estimate $R_\rmp^{\,(\varepsilon,N_0)}$ of the residue at the
given fixed values of $\varepsilon$ and $N_0$ can then be obtained
as the sample mean of $R_\rmp^{\,(\varepsilon,N_0)}(n)$ within this range of $n$--values,
in a way completely similar to what done in the case of the position of the pole.

\vskip 0.3cm

Now, the arguments given sofar are mainly qualitative, and therefore, 
the following problem emerges. \\[+5pt]
\noindent
\textbf{Problem.} \ How can the degree of approximation to $z_\rmp$ and $R_\rmp$ by the estimates
$z_\rmp^{\,(\varepsilon,N_0)}$ and $R_\rmp^{\,(\varepsilon,N_0)}$ be evaluated? \\[+5pt]
The discussion of this problem is given in the next subsection.

\subsection{Consistency relations for a meromorphic function, and measure of the degree
of approximation of the estimates to pole position and residue}
\label{subse:consistency:finite}

Referring to definition \eqref{pullo.1} of the approximate coefficients 
$c_{n,k}^{\,(\rmp;\varepsilon,N_0)}$, we can state the following auxiliary lemma.

\begin{lemma}
\label{lem:2}
For every fixed $k\in\N$, the following statements hold:
\begin{flalign}
(\mathrm{i})  &&&
\sum_{n=0}^\infty \left|c_{n,k}^{\,(\rmp;0,\infty)}\right|^2 = \|h_k(\rmi y)\|^2_{L^2(-\infty,+\infty)}
\doteq C_k \quad (C_k = \mathrm{const.}).\qquad\quad \label{55.3} \\
(\mathrm{ii}) &&&
\sum_{n=0}^\infty \left|c_{n,k}^{\,(\rmp;\varepsilon,N_0)}\right|^2=+\infty 
\qquad (\varepsilon>0, N_0 < \infty). \label{55.4} \\
(\mathrm{iii})&&&
\lim_{\substack{{N_0\rightarrow+\infty} \\ {\varepsilon\rightarrow 0}}}
c_{n,k}^{\,(\rmp;\varepsilon,N_0)}= c_{n,k}^{\,(\rmp;0,\infty)}= c_{n,k}^{(\rmp)} 
\qquad (n\in\N). \label{55.5}
\end{flalign}
{\rm(iv)} Let $m^{(\rmp)}(\varepsilon,N_0;k)$ be defined as
\beq
m^{(\rmp)}(\varepsilon,N_0;k) \doteq \max\left\{m\in\N :
\sum_{n=0}^{m} \left|c_{n,k}^{\,(\rmp;\varepsilon,N_0)}\right|^2
\leqslant C_k\right\},
\label{55.6}
\eeq
then
\beq
\lim_{\substack{{N_0\rightarrow+\infty} \\ {\varepsilon\rightarrow 0}}}
m^{(\rmp)}(\varepsilon,N_0;k)=+\infty \qquad (k\in\N).
\label{55.7}
\eeq
{\rm (v)} The sum
\beq
M_k^{(\rmp;\varepsilon,N_0)}(m) \doteq \sum_{n=0}^m
\left|c_{n,k}^{\,(\rmp;\varepsilon,N_0)}\right|^2 \qquad (k\in\N)
\label{55.8}
\eeq
satisfies the following properties:
\begin{enumerate}\addtolength{\itemsep}{0.4\baselineskip}
\item[\rm (v.a)] it does not decrease for increasing values of $m$;
\item[\rm (v.b)] for every $k\in\N$ the following asymptotic relationship holds:
\beq
\begin{split}
M_k^{(\rmp;\varepsilon,N_0)}(m) \geqslant
\left|c_{m,k}^{\,(\rmp;\varepsilon,N_0)}\right|^2
\staccrel{\sds\sim}{m\gg 1} A_k^{(\varepsilon,N_0)}
\cdot \,
& (2m)^{2\max\left\{N_0,\Real z_\rmp-\frac{1}{2}\right\}} \\
&
(\varepsilon>0 ~\mathrm{and}~ N_0<\infty ~\mathrm{fixed}),
\label{55.9}
\end{split}
\eeq
$A_k^{(\varepsilon,N_0)}$ being a quantity independent of $m$.
\end{enumerate}
\end{lemma}

\begin{proof}
(i) Since $c_{n,k}^{(\rmp;0,\infty)}=c_{n,k}^{(\rmp)}$, the statement follows 
from expansion \eqref{4.88.bis} and Parseval's theorem.
(ii) From the asymptotic expression in \eqref{pullo.3} it follows:
\beq
\left|c_{n,k}^{\,(\rmp;\varepsilon,N_0)}\right|
\staccrel{\sds\sim}{n ~\mathrm{sufficiently ~large}}
A_k^{(\varepsilon,N_0)}\,\cdot (2n)^{\max\{N_0,\Real z_\rmp-1/2\}}
\xrightarrow[n\to+\infty]{}+\infty,
\label{55.15bis}
\eeq
since $A_k^{(\varepsilon,N_0)}>0$, $N_0>0$.
Then, statement (ii) follows.
(iii) Consider the difference
\beq
c_{n,k}^{\,(\rmp;0,\infty)}-c_{n,k}^{\,(\rmp;\varepsilon,N_0)}
=\left(\oc_{n,k}^{\,(0,\infty)}-\oc_{n,k}^{\,(\varepsilon,N_0)}\right).
\label{55.16}
\eeq
For $N_0\to\infty$ and $\varepsilon\to 0$, the r.h.s. of \eqref{55.16} 
vanishes by Lemma \ref{lem:1}, and statement (iii) follows.
(iv) Define $m_1(\varepsilon,N_0;k)\doteq m^{(\rmp)}(\varepsilon,N_0;k)+1$. From
definition \eqref{55.6} it follows that
$\sum_{n=0}^{m_1}|c_{n,k}^{\,(\rmp;\varepsilon,N_0)}|^2 > C_k$.
For our purpose it is sufficient to prove that
$\lim_{\substack{{N_0\rightarrow+\infty}\\{\varepsilon\rightarrow 0}}}
m_1(\varepsilon,N_0;k)=+\infty$. Suppose, instead,
that such a limit is finite. Then, there exists a finite
number $m_*(k)$ (independent of $\varepsilon$ and $N_0$) such that
$\lim_{\substack{{N_0\rightarrow+\infty}\\{\varepsilon\rightarrow 0}}}
m_1(\varepsilon,N_0;k)\leqslant m_*(k)$. Then, we would have
\beq
C_k < \sum_{n=0}^{m_1(\varepsilon,N_0;k)}\left|c_{n,k}^{\,(\rmp;\varepsilon,N_0)}\right|^2
\leqslant
\sum_{n=0}^{m_*(k)}\left|c_{n,k}^{\,(\rmp;\varepsilon,N_0)}\right|^2.
\label{55.21}
\eeq
But, as $N_0\to+\infty$ and $\varepsilon\to 0$ we have (see also \eqref{55.5}):
\beq
C_k < \sum_{n=0}^{m_*(k)}\left|c_{n,k}^{\,(\rmp;0,\infty)}\right|^2
< \sum_{n=0}^{\infty}\left|c_{n,k}^{\,(\rmp;0,\infty)}\right|^2 = C_k,
\label{55.22}
\eeq
which is a contradiction. Then, statement (iv) is proved.
Statement (v.a) is obvious. Statement (v.b) follows from the
asymptotic behavior of the polynomials $P_m(-\rmi z)$ for large
values of $m$ and $z$ fixed, given in the Appendix 
(see formula \eqref{B3} and statement (ii)).
\end{proof}

\begin{corollary}
\label{cor:2}
From statements (iv) and (v) of Lemma \ref{lem:2} it
follows that, for $N_0$ sufficiently large, $\varepsilon$
sufficiently small and for every $k\in\N$,
the sum $M_k^{(\rmp;\varepsilon,N_0)}(m)$ exhibits (as a function of $m$) a \emph{plateau}:
i.e., a range of $m$--values where it is nearly constant. 
The upper limit of this range is given by $m^{(\rmp)}(\varepsilon,N_0;k)$. 
\end{corollary}

\begin{remark}
The plateau mentioned in Corollary \ref{cor:2} refers to the evaluation of
$M_k^{(\rmp;\varepsilon,N_0)}(m)$ against $m$ ($k\in\N$) and should not be confused
with the \emph{range of convergence} mentioned at the end of the previous subsection,
which refers to the evaluation of $z_\rmp^{(\varepsilon,N_0)}$ and $R_\rmp^{(\varepsilon,N_0)}$.
\label{rem:3}
\end{remark}

\smallskip

\noindent
Next, we introduce the sum defined by
\beq
\label{NN5bis}
\oM_k^{(\varepsilon,N_0)}(m) \doteq \sum_{n=0}^m \left|\oc_{n,k}^{\,(\varepsilon,N_0)}\right|^2
\qquad (k\in\N;\varepsilon>0,N_0<\infty),
\eeq
where the coefficients $\oc_{n,k}^{\,(\varepsilon,N_0)}$ are given in formula \eqref{4.10abis}
(restricted to $k\in\N$). The following two cases are worth being discussed: 

(1) Suppose that the function $f(z)$ being analyzed is analytic in $\Real z>0$.
In this case the $\oc_{n,k}$ (see \eqref{4.10a}) represent the expansion coefficients of the 
function $h_k(\rmi y)$ (analytic in $\Real z>0$) on the basis $\{\psi_n\}$
(see also formula \eqref{4.3}, where the sum on the r.h.s. coincides with $\oc_{n,k}$).
Therefore the terms $\oc_{n,k}$ and $\oc_{n,k}^{\,(\varepsilon,N_0)}$ enjoy properties
completely analogous to those established in Lemma \ref{lem:2} for the
coefficients $c_{n,k}^{\,(\rmp;\varepsilon,N_0)}$. 
In particular, the sum $\oM_k^{(\varepsilon,N_0)}(m)$ is expected to exhibit
(for $\varepsilon$ sufficiently small and $N_0$ sufficiently large) a \emph{plateau}
whose upper limit will be denoted by $\om(\varepsilon,N_0;k)$.
The properties of this plateau are strictly analogous to those
stated by Corollary \ref{cor:2}. 

(2) If $f(z)$ is meromorphic in $\Real z>0$, the coefficients
$\oc_{n,k}$ are not the expansion coefficients of the (meromorphic)
function $h_k(\rmi y)$ on the basis $\{\psi_n\}$, the actual coefficients
being given instead by formula \eqref{4.3}. Consequently, the sum
$\oM_k^{(\varepsilon,N_0)}(m)$ no longer satisfies the properties
mentioned in the previous point (1), in view of the absence of the pole term in the $\oc_{n,k}$.
Then, $\oM_k^{(\varepsilon,N_0)}(m)$ (as a function of $m$)
ought not to exhibit a \emph{plateau} in the neighborhood of a certain value of $m$. 

Summarizing, the analysis, as a function of $m$, of the sum $\oM_k^{(\varepsilon,N_0)}(m)$ 
(which can be computed from the input data set) can be exploited as an initial test 
of analyticity for the function under consideration.

\vskip 0.3cm

We can now proceed to give an answer to the problem posed in the previous subsection. 
The main idea is inspired by the consistency relations \eqref{4.2} (along with formula \eqref{4.3} 
for the coefficients), which make explicit the mutual relations among the pole parameters and 
the function samples. Equations \eqref{4.2} suggest to compare the samples of the input data set 
$\{f_k^{(\varepsilon)}\}_{k=0}^{N_0}$ with the corresponding values, denoted 
$\{\hf_k^{\,(\rmp;\varepsilon,N_0)}\}_{k=0}^{N_0}$, which can be actually computed when
the estimates $z_\rmp^{(\varepsilon,N_0)}$ and $R_\rmp^{(\varepsilon,N_0)}$ of pole
position and residue have been evaluated. Therefore, we are led to define the following 
approximate coefficients:
\begin{align}
\begin{split}
\hf_k^{\,(\rmp;\varepsilon,N_0)}
& \doteq (-1)^{k+1} \left\{\sum_{N=0}^{N_0} (-1)^N f_N^{(\varepsilon)} \, (1-\delta_{Nk})
-\frac{\pi R_\rmp^{(\varepsilon,N_0)}}{\cos\pi z_\rmp^{(\varepsilon,N_0)}} \right. \\
&\quad \left.-\ud\int_{-\infty}^{+\infty}\frac{\sum_{n=0}^{\hm^{\,(\rmp)}}
\hc_{n,k}^{\,(\rmp;\varepsilon,N_0)}\psi_n(y)}{[(x+\ud)-\rmi y]\cosh\pi y}\,\rmd y \right\} \\
&= (-1)^{k+1} \left\{\sum_{N=0}^{N_0} (-1)^N f_N^{(\varepsilon)} \, (1-\delta_{Nk})
-\frac{\pi R_\rmp^{(\varepsilon,N_0)}}{\cos\pi z_\rmp^{(\varepsilon,N_0)}} \right. \\
&\quad\left.+\sum_{n=0}^{\hm^{(\rmp)}} \hc_{n,k}^{\,(\rmp;\varepsilon,N_0)} 
Q_n\left[-\rmi\left(k+\ud\right)\right]\right\}\quad (k=0,\ldots,N_0),
\end{split}
\label{55.23} \\
\intertext{where:}
\begin{split}
\hc_{n,k}^{\,(\rmp;\varepsilon,N_0)} & \doteq 2\sqrt{\pi}
\left\{\sum_{N=0}^{N_0}\frac{(-1)^N}{N!}(N-k)f_N^{(\varepsilon)}
\,P_n\left[-\rmi\left(N+\frac{1}{2}\right)\right]\right.\\
& \quad\left.-\left(z_\rmp^{(\varepsilon,N_0)}-k-\frac{1}{2}\right)
R_\rmp^{(\varepsilon,N_0)}\,
\Gamma\left(\frac{1}{2}-z_\rmp^{(\varepsilon,N_0)}\right)
\,P_n\left(-\rmi z_\rmp^{(\varepsilon,N_0)}\right)\right\} \\
& = \oc_{n,k}^{(\varepsilon,N_0)}-\left(\z_\rmp^{(\varepsilon,N_0)}-k\right)\tau_n^{(\varepsilon,N_0)}
\qquad (n=0,1,2,\ldots),
\end{split}
\label{4.3.bis}
\end{align}
$\psi_n(y)$ being defined in \eqref{0.8.bis},
$Q_n\left[-\rmi\left(k+\ud\right)\right]$ in \eqref{0.18.new.1}, while the truncation number 
$\hm^{(\rmp)}=\hm^{(\rmp)}(\varepsilon,N_0;k)$ will be defined in what follows. 
Notice that the coefficients $\hc_{n,k}^{\,(\rmp;\varepsilon,N_0)}$ differ from the corresponding 
coefficients $c_{n,k}^{\,(\rmp;\varepsilon,N_0)}$ (see \eqref{pullo.1})
for the fact that in the expression of $\hc_{n,k}^{\,(\rmp;\varepsilon,N_0)}$ use is made of
the computed estimates $z_\rmp^{(\varepsilon,N_0)}$ and $R_\rmp^{(\varepsilon,N_0)}$
instead of their exact values, which are unknown as long as $\varepsilon>0$ and $N_0<\infty$. 
However, if the computed estimates $z_\rmp^{(\varepsilon,N_0)}$ and $R_\rmp^{(\varepsilon,N_0)}$ are
``close'' to the exact values $z_\rmp$ and $R_\rmp$, then correspondingly the terms 
$\hc_{n,k}^{\,(\rmp;\varepsilon,N_0)}$ are ``close'' to the coefficients 
$c_{n,k}^{\,(\rmp;\varepsilon,N_0)}$. Hence, in strict analogy with the definitions given
in Lemma \ref{lem:2} concerning the coefficients $c_{n,k}^{\,(\rmp;\varepsilon,N_0)}$, we can introduce
the following sum (see \eqref{55.8}) :
\beq
\hM_k^{\,(\rmp;\varepsilon,N_0)}(m)
\doteq \sum_{n=0}^m \left|\hc_{n,k}^{\,(\rmp;\varepsilon,N_0)}\right|^2
\qquad (k\in\N).
\label{L6}
\eeq
Notice that $\hM_k^{\,(\rmp;0,\infty)}(m)=M_k^{\,(\rmp;0,\infty)}(m)$
since $\hc_{n,k}^{\,(\rmp;0,\infty)}=c_{n,k}^{\,(\rmp;0,\infty)}$.

The analysis of $\hM_k^{\,(\rmp;\varepsilon,N_0)}(m)$ may
provide us with a first qualitative test of the degree of approximation of
$z_\rmp^{(\varepsilon,N_0)}$ and $R_\rmp^{(\varepsilon,N_0)}$.
In fact, if $\hM_k^{\,(\rmp;\varepsilon,N_0)}(m)$
exhibits a \emph{plateau} as a function of $m$, that is, if it manifests a behavior analogous 
to that expected from the analysis of $M_k^{\,(\rmp;\varepsilon,N_0)}(m)$ against $m$ 
(the existence of a \emph{plateau} in $M_k^{\,(\rmp;\varepsilon,N_0)}(m)$ is proved by 
Corollary \ref{cor:2}), then the coefficients
$\hc_{n,k}^{\,(\rmp;\varepsilon,N_0)}$ are likely to be close to the
corresponding values $c_{n,k}^{\,(\rmp;\varepsilon,N_0)}$, and,
accordingly, the estimates $z_\rmp^{(\varepsilon,N_0)}$ and $R_\rmp^{(\varepsilon,N_0)}$ 
are close to their exact values $z_\rmp$ and $R_\rmp$.
In this case, the upper limit of this \emph{plateau}, which we denote by
$\hm^{(\rmp)}(\varepsilon,N_0;k)$, enjoys properties strictly analogous to those of
$m^{(\rmp)}(\varepsilon,N_0;k)$, which have been stated in Lemma \ref{lem:2},
and therefore, can be used as the truncation number for the rightmost sum in formula \eqref{55.23}.

Now, once $\hm^{(\rmp)}(\varepsilon,N_0;k)$ has been set, formula \eqref{55.23}
allows us to compute the \emph{approximate samples} $\hf_k^{\,(\rmp;\varepsilon,N_0)}$,
representing the set of \emph{samples} which are compatible with the computed 
estimates of pole position $z_\rmp^{(\varepsilon,N_0)}$ and residue $R_\rmp^{(\varepsilon,N_0)}$.
At this point we are naturally brought to compare these \emph{approximate samples} 
$\hf_k^{\,(\rmp;\varepsilon,N_0)}$ with the input data, i.e. the \emph{noisy samples} 
$f_k^{(\varepsilon)}$.
Then, as a measure of the accuracy of the computation
of the samples, and hence also of the pole parameters, use can be made of the relative 
root mean squared error:
\beq
\label{num1}
\delta^{(\varepsilon,N_0)} \doteq \left(\frac{1}{(N_0+1)}
\sum_{k=0}^{N_0}\frac{\left|f_k^{(\varepsilon)}-\hf_k^{\,(\rmp;\varepsilon,N_0)}\right|^2}
{\left|f_k^{(\varepsilon)}\right|^2}\right)^{1/2},
\eeq
which gives a quantitative numerical evaluation of the degree of approximation to $z_\rmp$
and $R_\rmp$ by the estimates $z_\rmp^{(\varepsilon,N_0)}$ and $R_\rmp^{(\varepsilon,N_0)}$,
as initially required by the problem in Section \ref{subse:algorithm:finite}.

\begin{remark}
In the case of functions $f(z)$ analytic in $\Real z >0$, the formula for computing the
\emph{approximate samples} is obviously different for the absence of the pole term.
Then, in formula \eqref{num1}, instead of $\hf_k^{\,(\rmp;\varepsilon,N_0)}$ given in \eqref{55.23},
there must be inserted the terms $\hf_k^{\,(\varepsilon,N_0)}$ given by:
\begin{align}
\begin{split}
\hf_k^{\,(\varepsilon,N_0)}
& \doteq (-1)^{k+1}\left\{\sum_{N=0}^{N_0} (-1)^N f_N^{(\varepsilon)} \, (1-\delta_{Nk}) \right. \\
&\quad\left.+\sum_{n=0}^{\om(\varepsilon,N_0;k)} \hc_{n,k}^{\,(\varepsilon,N_0)} 
Q_n\left[-\rmi\left(k+\ud\right)\right]\right\}\quad (k=0,\ldots,N_0),
\end{split}
\label{M1} \\
\intertext{where, for $n=0,1,2,\ldots$}
\begin{split}
\hc_{n,k}^{\,(\varepsilon,N_0)} & \doteq 2\sqrt{\pi}
\sum_{N=0}^{N_0}\frac{(-1)^N}{N!}(N-k)f_N^{(\varepsilon)}
\,P_n\left[-\rmi\left(N+\frac{1}{2}\right)\right] \equiv \oc_{n,k}^{(\varepsilon,N_0)},
\end{split}
\label{M1.bis}
\end{align}
$\om(\varepsilon,N_0;k)$ being the upper limit of the plateau exhibited by 
$\oM_k^{(\varepsilon,N_0)}(m)$, regarded as a function of $m$.
\label{rem:4}
\end{remark}

\begin{remark}
In the case $f(z)$ is a meromorphic function, 
for all values $m \lesssim \hm^{(\rmp)}(\varepsilon,N_0;k)$
such that $\hM_k^{\,(\rmp;\varepsilon,N_0)}(m)$ 
exhibits a \emph{plateau} we have $\hc_{n,k}^{\,(\rmp;\varepsilon,N_0)}\sim 0$.
It then follows that the actual value of the truncation number in the rightmost
sum on the r.h.s. of formula \eqref{55.23}
is not critical, and therefore any such $m \lesssim \hm^{(\rmp)}(\varepsilon,N_0;k)$
may be selected as an acceptable value where to stop the sum. 
In the next section, devoted to numerical examples,
we will see that, in the actual numerical implementation, it can be convenient 
to truncate the sum
at a $m$--value slightly different from $\hm^{(\rmp)}(\varepsilon,N_0;k)$. 
The practical evaluation of this truncation point, denoted $\hm^{(\rmp)}_\rmt$,
will be specified in the next section. \\
Similar arguments hold \emph{mutatis mutandis} in the case $f(z)$ is an analytic function,
the role of $\hM_k^{\,(\rmp;\varepsilon,N_0)}(m)$, $\hm^{(\rmp)}(\varepsilon,N_0;k)$,
$\hc_{n,k}^{\,(\rmp;\varepsilon,N_0)}$, and $\hm^{\,(\rmp)}_\rmt$ being now played by
$\oM_k^{(\varepsilon,N_0)}(m)$, $\om(\varepsilon,N_0;k)$, 
$\oc_{n,k}^{\,(\varepsilon,N_0)}$, and $\om_\rmt$, respectively.
\label{rem:5}
\end{remark}

\begin{remark}
In the case $f(z)$ is a meromorphic function, and
in view of what has been discussed above, the formulae for the actual numerical implementation 
of the interpolation formula \eqref{3.1} (along with \eqref{3.2}) read:
\beq
\begin{split}
\hf^{\,(\varepsilon,N_0)}\left(x+\frac{1}{2}\right) &= \sum_{N=0}^{N_0} f_N^{(\varepsilon)} \sinc(x-N)
-\frac{R_\rmp^{(\varepsilon,N_0)}}{\cos\left(\pi z_\rmp^{(\varepsilon,N_0)}\right)}
\frac{\sin\pi x}{\left(x+\ud-z_\rmp^{(\varepsilon,N_0)}\right)} \\
& \,\,\, -\frac{\sin\pi x}{\pi}\sum_{n=0}^{\ \hm^{(\rmp)}}
\hc_{n}^{\,(\rmp;\varepsilon,N_0)}\,Q_n\left[-\rmi\left(x+\ud\right)\right]
\qquad \left(x > -\frac{1}{2}\right),
\label{3.1.bis}
\end{split}
\eeq
where, for $n=0,1,2,\ldots$
\beq
\begin{split}
\hc_n^{\,(\rmp;\varepsilon,N_0)} &=
2\sqrt{\pi}\left\{\sum_{N=0}^{N_0} \frac{(-1)^N}{N!} f_N^{(\varepsilon)}\, 
P_n\left[-\rmi\left(N+\frac{1}{2}\right)\right] \right. \\
& \quad\left. -R_\rmp^{(\varepsilon,N_0)}\Gamma\left(\frac{1}{2}-z_\rmp^{(\varepsilon,N_0)}\right) 
P_n\left(-\rmi z_\rmp^{(\varepsilon,N_0)}\right)\right\},
\label{3.2.bis}
\end{split}
\eeq
and, for the given values of $\varepsilon$ and $N_0$ and for every $x>-\ud$, 
$\hf^{\,(\varepsilon,N_0)}(x+\ud)$ represents the approximation to the function $f(x+\ud)$.
\label{rem:6}
\end{remark}

\section{Numerical examples}
\label{se:numerical}

The purpose of this section is to illustrate through numerical examples the main steps of the theory.
To begin with, we consider as a preparatory example 
the function $f_1(z)=C/(z+5)^5$, ($C$ constant), which satisfies the conditions 
assumed in Theorem \ref{the:1new} and,
in particular, it is analytic in $\Real z>0$; the analysis is summarized in 
Figs. \ref{fig:2} and \ref{fig:3}.
In Fig. \ref{fig:2}a the plot of the sum $\oM_k^{(\varepsilon,N_0)}(m)$ (defined in \eqref{NN5bis}),
computed for various values of $k$ (see the figure legend for numerical details),
displays clearly the presence of the plateaux, which manifest the
analyticity of the function $f_1(z)$ in $\Real z>0$ (see also Corollary \ref{cor:2}).
In this example, as well as in all those that follow (with the exception 
of that referring to Fig. \ref{fig:8}), we set $\varepsilon\equiv\varepsilon_\mathrm{R}$, which means that
no noise has been added to the input samples $f_N$, and that the only source of 
(inevitable) noise is given by the numerical roundoff error. The plateaux range 
approximately from $n_\mathrm{min}(\varepsilon,N_0;k) \sim 40$ through 
$\om (\varepsilon,N_0;k) \sim 240$; for $m \gtrsim \om(\varepsilon,N_0;k)$ we see that 
$\oM_k^{(\varepsilon,N_0)}(m)$ starts to diverge as a power of $m$ (see \eqref{55.9}) for
the presence of the roundoff noise and the finiteness of the number of input samples 
(in this case $N_0=60$). From the inspection of these plateaux we can determine the 
truncation number $\om_\rmt$, which must lie within the plateaux, and which is necessary for 
the reconstruction of the data samples (see \eqref{M1}).
The choice of $\om_\rmt$ within the plateaux is not critical
for the accuracy of the final result (see Fig. \ref{fig:3}a);
the actual value of $\om_\rmt$ (within the plateau) can be conveniently set by exploiting  
formula \eqref{M1} for the sample reconstruction.
In fact, by \eqref{M1} we can compute the approximate samples 
$\hf_k^{\,(\varepsilon,N_0)}$ (which depend on the truncation number $\om_\rmt$)
and, correspondingly, also the mean error $\delta^{(\varepsilon,N_0)}$ (see \eqref{num1}).
The strategy is then to set $\om_\rmt$ as the integral number which minimizes 
$\delta^{(\varepsilon,N_0)}$ ($\varepsilon$ and $N_0$ being fixed).
In this example, the value of $\om_\rmt$ which minimizes $\delta^{(\varepsilon,N_0)}$ 
(with $\varepsilon=\varepsilon_\mathrm{R}$ and $N_0=60$) is $\om_\rmt(\varepsilon,N_0;k)=122$; the
reconstructed samples, computed through \eqref{M1}, are shown in 
Fig. \ref{fig:2}b (filled dots) superimposed to the function $f_1(x)$ (solid line); 
the high quality of the reconstruction ($\delta^{(\varepsilon,N_0)}=2.74\times 10^{-5}$) is evident.

The role played by the various parameters intervening in the algorithm is summarized in 
Fig. \ref{fig:3}. In Fig. \ref{fig:3}a we show the behavior of the reconstruction error 
$\delta^{(\varepsilon,N_0)}$ as a function of the truncation number $\om_\mathrm{t}$ (see \eqref{M1}). 
We see that the error becomes tiny when $\om_\mathrm{t}$ enters the plateaux ($\om_\rmt\sim 50$, see 
Fig. \ref{fig:2}a), and does not vary appreciably as long as $\om_\rmt$ remains within it.
Figure \ref{fig:3}b shows the sum $\oM_{k}^{(\varepsilon,N_0)}(m)$ against $m$ ($k=10$) for various values 
of the number of input samples $N_0$. It can be seen that the length of the plateaux 
increases as the number of input data increases, reflecting the increase of \emph{information} 
available for the computation. At $N_0=60$ the effect of the roundoff noise appears, and it is 
this noise indeed which limits superiorly the length of the plateau; in fact, for this function, 
the length of the plateau no longer increases for $N_0\gtrsim 60$.
In Fig. \ref{fig:3}c we investigate how the analysis depends on the asymptotic behavior 
(for $z\rightarrow\infty$) of the function $f(z)$. For this purpose we consider the function 
$f_1^{(q)}(z)=C/(z+5)^q$, and plot in Fig. \ref{fig:3}c the sum $\oM_{k}^{(\varepsilon,N_0)}(m)$
($k=10$, $\varepsilon=\varepsilon_\mathrm{R}$, and $N_0=60$ fixed) at different values of $q$. 
Even in this case the length of the 
plateaux increases evidently with $q$. It should be remarked that, as expected, for $q=1$
no plateau appears since the function 
$h_k(\rmi y)=(\rmi y-k-\frac{1}{2})f_1^{(1)}(\rmi y) \not\in L_2(-\infty,+\infty)$
(see Theorem \ref{the:3new}).
\begin{figure}[htb] 
\captionsetup{width=0.95\textwidth}
\begin{center}
\leavevmode
\includegraphics[width=11.4cm]{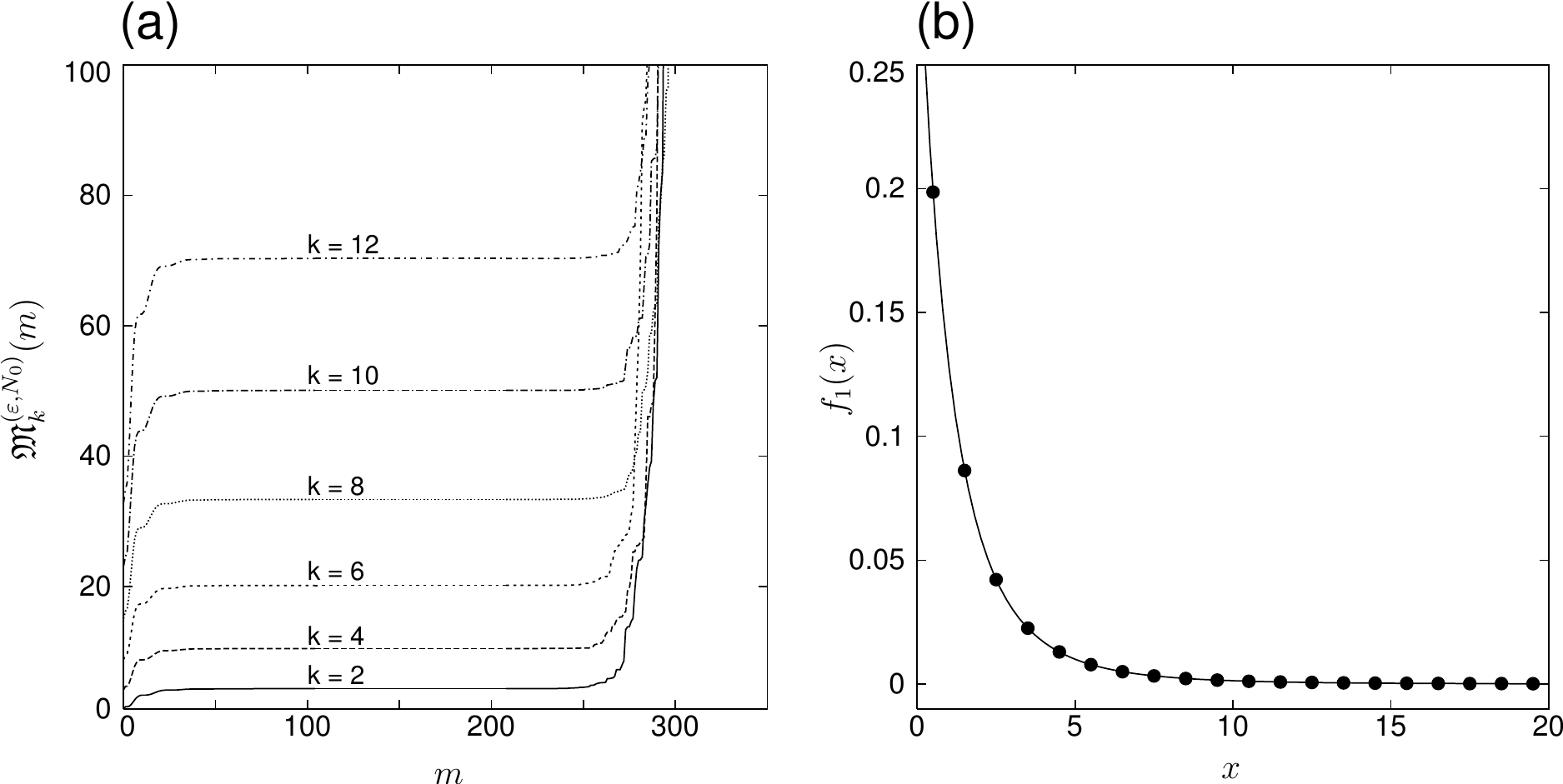}
\caption{\label{fig:2} 
\small \textsl{Analysis of a function analytic in
$\Real z>0$: $f_1(z)=C(z+5)^{-5}$. $C=10^{3}$; $N_0=60$;
$\varepsilon \equiv \varepsilon_\mathrm{R}$: i.e., no noise is added to the function samples
but only numerical roundoff noise is present. {\bf (a)}: Plot of the
sum $\oM_k^{(\varepsilon,N_0)}(m)$ vs. $m$ for some values of $k$ (see
formula \eqref{NN5bis}). {\bf (b)}: Reconstructed samples $\hf_j^{\,(\varepsilon,N_0)}$ 
(filled dots) of the function $f_1(x)$ (solid line) at 20 half--integers
values $x_j\doteq j+\frac{1}{2}$, $j=0,\ldots,19$, computed by using 
Eq. \eqref{M1}. For every $k$, the truncation number has been set to the value
$\om_\rmt(\varepsilon,N_0;k)=122$, which represents the value
that minimizes the relative root mean squared error $\delta^{(\varepsilon,N_0)}$ 
(see Eq. \eqref{num1}): in this case, $\delta^{(\varepsilon,N_0)} = 2.74 \times 10^{-5}$.}
}
\end{center}
\end{figure}
Parallelly, in Fig. \ref{fig:3}d, where the plots of the 
reconstructed samples $\hf_k^{\,(\varepsilon,N_0)}$ are shown for various values of $q$, 
we see that for $q=1$ the reconstruction clearly deteriorates.

We can now move on to analyze the case of meromorphic functions. In order to elucidate 
the algorithm presented in Sections \ref{subse:algorithm:infinite} and \ref{subse:algorithm:finite}, 
we begin by considering a simple, though canonical, meromorphic function:
$f_2(z)=\frac{C}{(z+3)^3}\frac{1}{(z-z_\rmp)}$, which has a first order pole in the 
half--plane $\Real z>0$ at $z_\rmp=6.2+0.15\rmi$; the factor $(z+3)^{-3}$ in $f_2(z)$
guarantees the necessary asymptotic behavior for $\Real z\to +\infty$, without introducing 
poles in the half--plane $\Real z>0$. The various steps, which lead to the recovery of pole 
location and residue, are summarized in Figs. \ref{fig:4} and \ref{fig:5}.

From the input data $(f_2)_N^{(\varepsilon_\mathrm{R})}$ ($N=0,\ldots,N_0$), which are the numerical 
approximation of the function samples $f_2(N+\ud)$, we can compute, for integral values of $k$,
the sum $\oM_k^{(\varepsilon,N_0)}(m)$ (see formulae \eqref{NN5bis} and \eqref{4.10abis}),
whose plot versus $m$ is shown in Fig. \ref{fig:4}a. The absence of any plateau 
in this plot indicates \emph{the lack of analyticity} of the function $f_2(z)$ 
in $\Real z>0$; in fact, as discussed in Section \ref{subse:consistency:finite} 
(see, in particular, Corollary \ref{cor:2}), it is instead the sum $M_k^{(\rmp;\varepsilon,N_0)}(m)$
(and its approximate version $\hM_k^{\,(\rmp;\varepsilon,N_0)}(m)$, see \eqref{55.8} and \eqref{L6}), 
containing the terms $(\z_\rmp-k)\tau_n$ related to the (unknown) pole, which is expected 
to exhibit a plateau. Thus, as discussed in 
Section \ref{subse:consistency:finite}, the study of $\oM_k^{(\varepsilon,N_0)}(m)$ can be used 
as a preliminary tool for testing the analyticity for $f(z)$. 
\begin{figure}[htb]  
\captionsetup{width=0.95\textwidth}
\begin{center}
\leavevmode
\includegraphics[width=11.4cm]{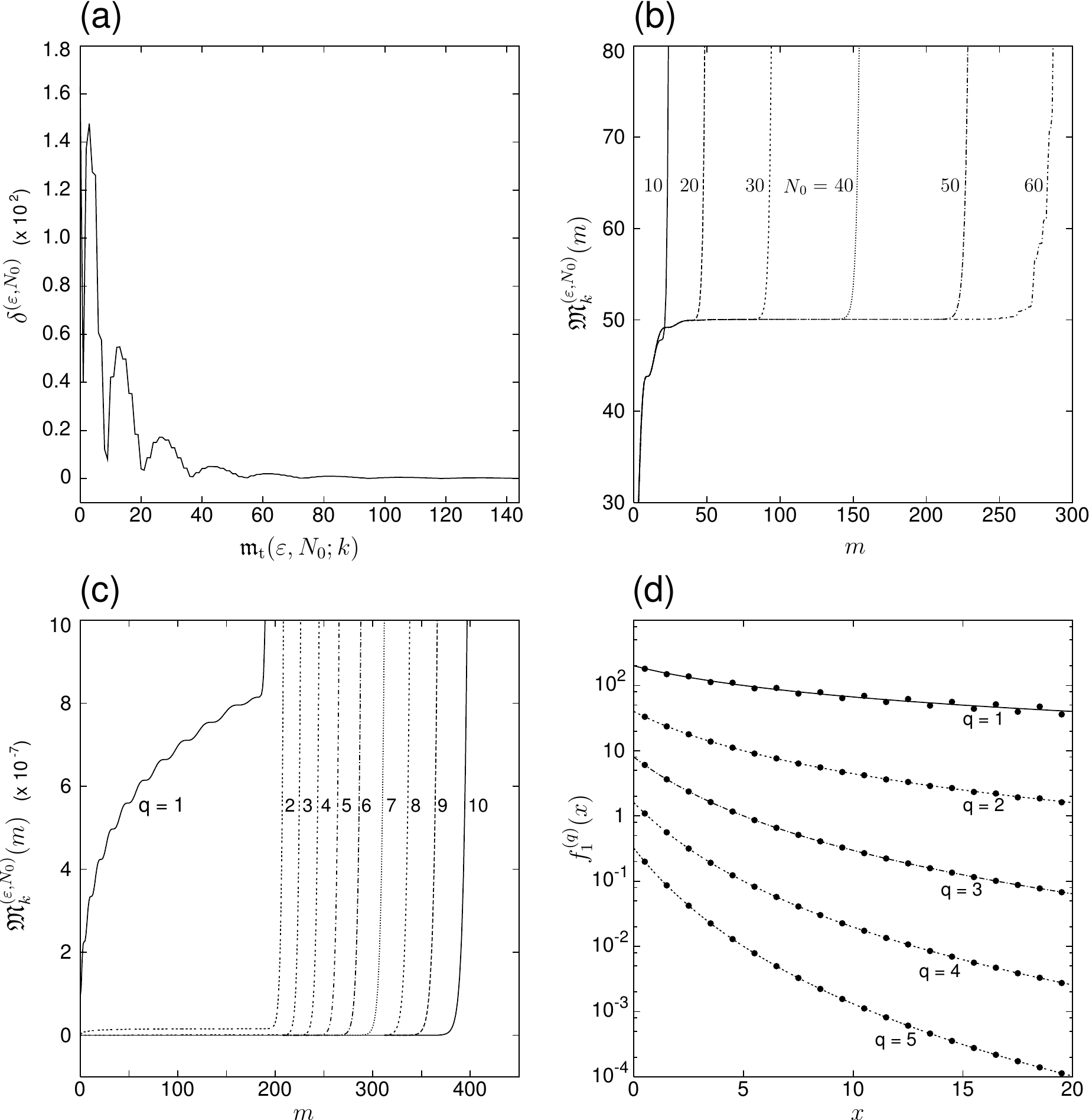}
\caption{\label{fig:3}
\small
\textsl{Analysis of a function analytic in $\Real z>0$: $f_1^{(q)}(z)=C(z+5)^{-q}$;
$C=10^{3}$; $N_0=60$; $\varepsilon \equiv \varepsilon_\mathrm{R}$.
{\bf (a)}: $q=5$. Relative root mean squared error $\delta^{(\varepsilon,N_0)}$ vs. the 
truncation number $\om_\rmt$ (see \eqref{num1}, \eqref{M1}, and the legend of Fig. \ref{fig:2}b).
{\bf (b)}: $q=5$. Plot of the sum $\oM_{k}^{(\varepsilon,N_0)}(m)$ vs. $m$, at various values of the 
number $N_0$ of input function samples: $N_0=10,20,30,40,50,60$ (see \eqref{NN5bis}, \eqref{4.10abis}); 
$k=10$.
{\bf (c)}: Plot of $\oM_{k}^{(\varepsilon,N_0)}(m)$ vs. $m$ for different values of the 
exponent $q$: $q=1,\ldots,10$; $k = 10$.
{\bf (d)}: Reconstructed samples $\hf_j^{\,(\varepsilon,N_0)}$ 
(filled dots) of the function $f_1^{(q)}(x)$ (solid line) at 
the half--integers $x_j$, computed by Eq. \eqref{M1} for various values of the exponent $q$.
For each $q$, the value of the truncation number $\om_\rmt(\varepsilon,N_0;k)$, used in 
formula \eqref{M1} for reconstructing the samples, 
is the one which minimizes $\delta^{(\varepsilon,N_0)}$;
$q=1:\om_\rmt(\varepsilon,N_0;k)=41,\delta^{(\varepsilon,N_0)}=2.08\times 10^{-2}$;
$q=2:\om_\rmt(\varepsilon,N_0;k)=72,\delta^{(\varepsilon,N_0)}=1.38\times 10^{-2}$;
$q=3:\om_\rmt(\varepsilon,N_0;k)=29,\delta^{(\varepsilon,N_0)}=1.08\times 10^{-3}$;
$q=4:\om_\rmt(\varepsilon,N_0;k)=138,\delta^{(\varepsilon,N_0)}=4.87\times 10^{-4}$;
$q=5:\om_\rmt(\varepsilon,N_0;k)=122,\delta^{(\varepsilon,N_0)}=2.74\times 10^{-5}$.}
}
\end{center}
\end{figure}
It is worth remarking that this test is not limited to functions with only polar singularities, but works also in the 
case of multivalued functions with branch cuts in $\Real z>0$; moreover, this procedure 
works also as a test when the function $f(z)$ has a pole (of order $>1$), whose residue is null
(for the sake of brevity, the numerical examples are not reported here).

Figure \ref{fig:4}b shows the plots of $\Real z_\rmp^{(\varepsilon,N_0)}(n)$ and
$\Imag z_\rmp^{(\varepsilon,N_0)}(n)$ (computed by means of Eq. \eqref{Q5.2} and recalling that
$z_\rmp^{\,(\varepsilon,N_0)}(n)=\z_\rmp^{\,(\varepsilon,N_0)}(n)+\ud$) 
versus $n$. 
\begin{figure}[htp] 
\captionsetup{width=0.95\textwidth}
\begin{center}
\leavevmode
\includegraphics[width=11.4cm]{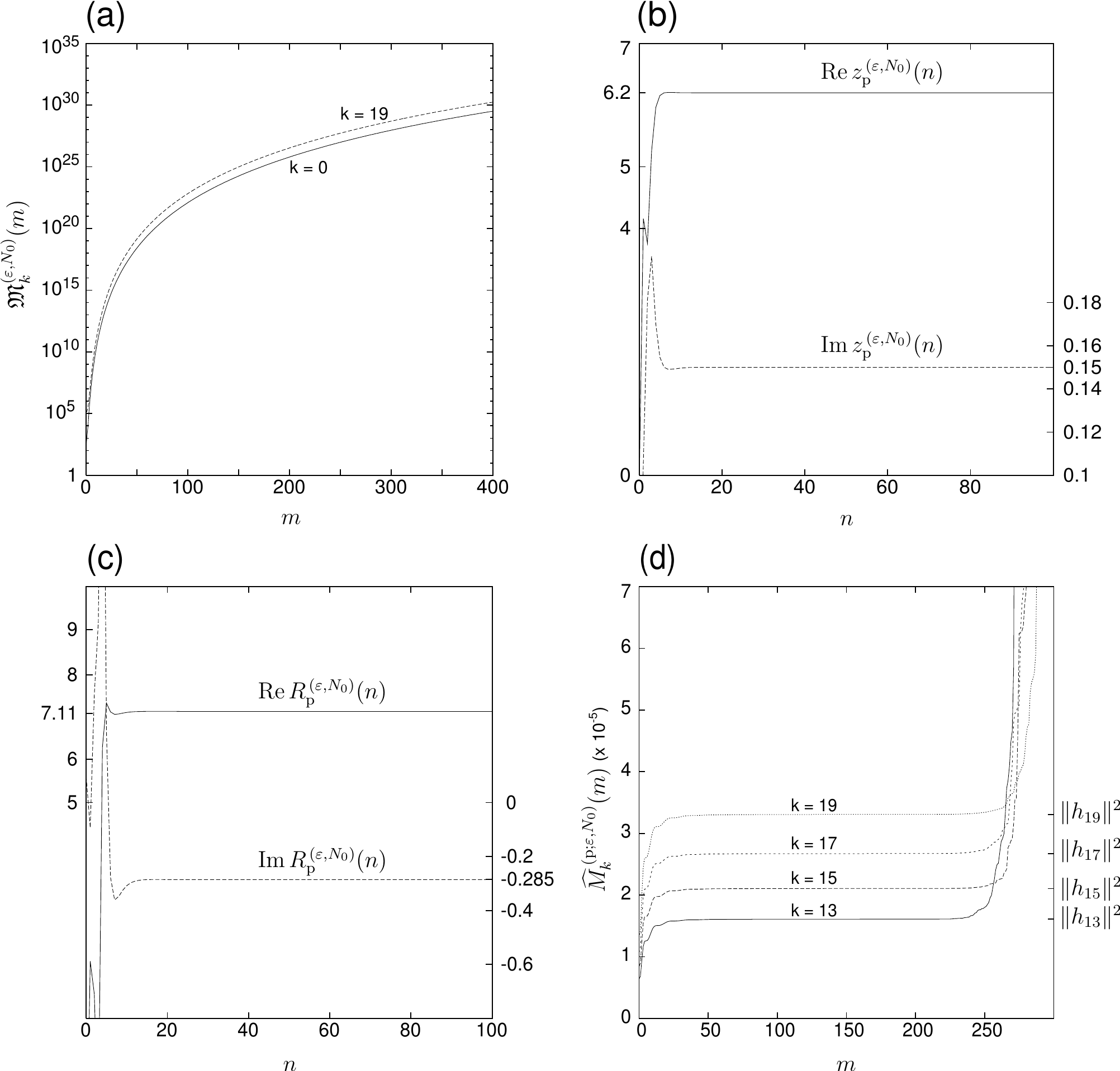}
\caption{\label{fig:4}
\scriptsize
\textsl{Analysis of a function with one first order pole in $\Real z>0$:
$f_2(z)=C(z+5)^{-3}(z-z_\rmp)^{-1}$. $z_\rmp=6.2+0.15\rmi$, $R_\rmp=7.110146-0.2858122\rmi$.
$C=10^4$, $N_0=60$, $\varepsilon\equiv\varepsilon_\mathrm{R}$.
{\bf (a)}: Analyticity test for $f_2(z)$ in $\Real z>0$: plot of the sum $\oM_k^{(\varepsilon,N_0)}(m)$
against $m$, at different values of $k$ (notice the absence of the pole term in \eqref{4.10abis}).
{\bf (b)}: Recovery of pole location. Real (left scale) and imaginary (right scale) part of the function 
$z_\rmp^{(\varepsilon,N_0)}(n)$ vs. $n$: for every value of $n$, $z_\rmp^{(\varepsilon,N_0)}(n)$ has been 
computed by means of formula \eqref{Q5.2} (see the algorithm described in 
Subsection \ref{subse:algorithm:finite}).
Length of the \emph{ranges of convergence}, obtained with bin--width $W_\mathrm{p} = 10^{-3}\%$;
Real part: $L_\mathrm{p}^\mathrm{R} = 416$; Imaginary part: $L_\mathrm{p}^\mathrm{I} = 418$.
The position of the pole, computed as the sample mean of $\Real z_\rmp^{(\varepsilon,N_0)}(n)$ for $n$ 
within the \emph{range of convergence}, is:
$z_\rmp^{(\varepsilon,N_0)}=(6.20000005 \pm 1.5 \times 10^{-7})+(1.4999996 \times 10^{-1} \pm 
1.9 \times 10^{-7})\rmi$.
{\bf (c)}: Recovery of the pole residue. Real (left scale) and imaginary (right scale) part of the 
function $R_\rmp^{(\varepsilon,N_0)}(n)$ vs. $n$: for every $n$, $R_\rmp^{(\varepsilon,N_0)}(n)$ 
has been computed by Eq. \eqref{L4}, using the value of $z_\rmp^{(\varepsilon,N_0)}$ estimated 
from the data shown in panel (b). 
Length of the \emph{ranges of convergence}, obtained with bin--width $W_\mathrm{p} = 10^{-3}\%$:
$L_\mathrm{p}^\mathrm{R} = 509$, $L_\mathrm{p}^\mathrm{I} = 454$.
The residue, computed as the sample mean of $R_\rmp^{(\varepsilon,N_0)}(n)$ for $n$ within the 
\emph{range of convergence} is:
$R_\rmp^{(\varepsilon,N_0)}=(7.1101458 \pm 7.4 \times 10^{-6})-(2.8581197 \times 10^{-1} 
\pm 5.6 \times 10^{-7})\rmi$.
{\bf (d)}: Plateaux after pole recovery. Plot of the sum $\hM_k^{\,(\rmp;\varepsilon,N_0)}(m)$ vs. $m$ 
for various $k$ (see formula \eqref{L6}, and cf. panel (a)). 
The value of the plateau of $\hM_k^{\,(\rmp;\varepsilon,N_0)}$
is in correspondence of the squared norm $\|h_k(\rmi y)\|^2$ (see \eqref{55.3}).
The values of $z_\rmp^{(\varepsilon,N_0)}$ and $R_\rmp^{(\varepsilon,N_0)}$
inserted in the coefficients in \eqref{pullo.1} (and used 
to calculate $\hM_k^{\,(\rmp;\varepsilon,N_0)}(m)$), 
are those obtained from the analysis of the data shown in panels (b) and (c).}
}
\end{center}
\end{figure}
In this example both $\varepsilon$ and $N_0$ must be considered as fixed, 
and take on the values 
$\varepsilon=\varepsilon_\mathrm{R}$ (which means that only numerical roundoff error 
is present) and $N_0=60$. In Fig. \ref{fig:4}b wild oscillations
can be observed in both plots for $n < n_\mathrm{min}\simeq 10$, whereas, 
for $n>n_\mathrm{min}$ the value of $z_\rmp^{(\varepsilon,N_0)}(n)$ 
(that is, of its real and imaginary parts) remains nearly constant over a wide range 
of $n$--values; the upper limit of this interval (not visible in this figure) is 
$n_\mathrm{max}=426$ (see Section \ref{subse:algorithm:finite}).

\begin{figure}[htb] 
\captionsetup{width=0.95\textwidth}
\begin{center}
\leavevmode
\includegraphics[width=11.4cm]{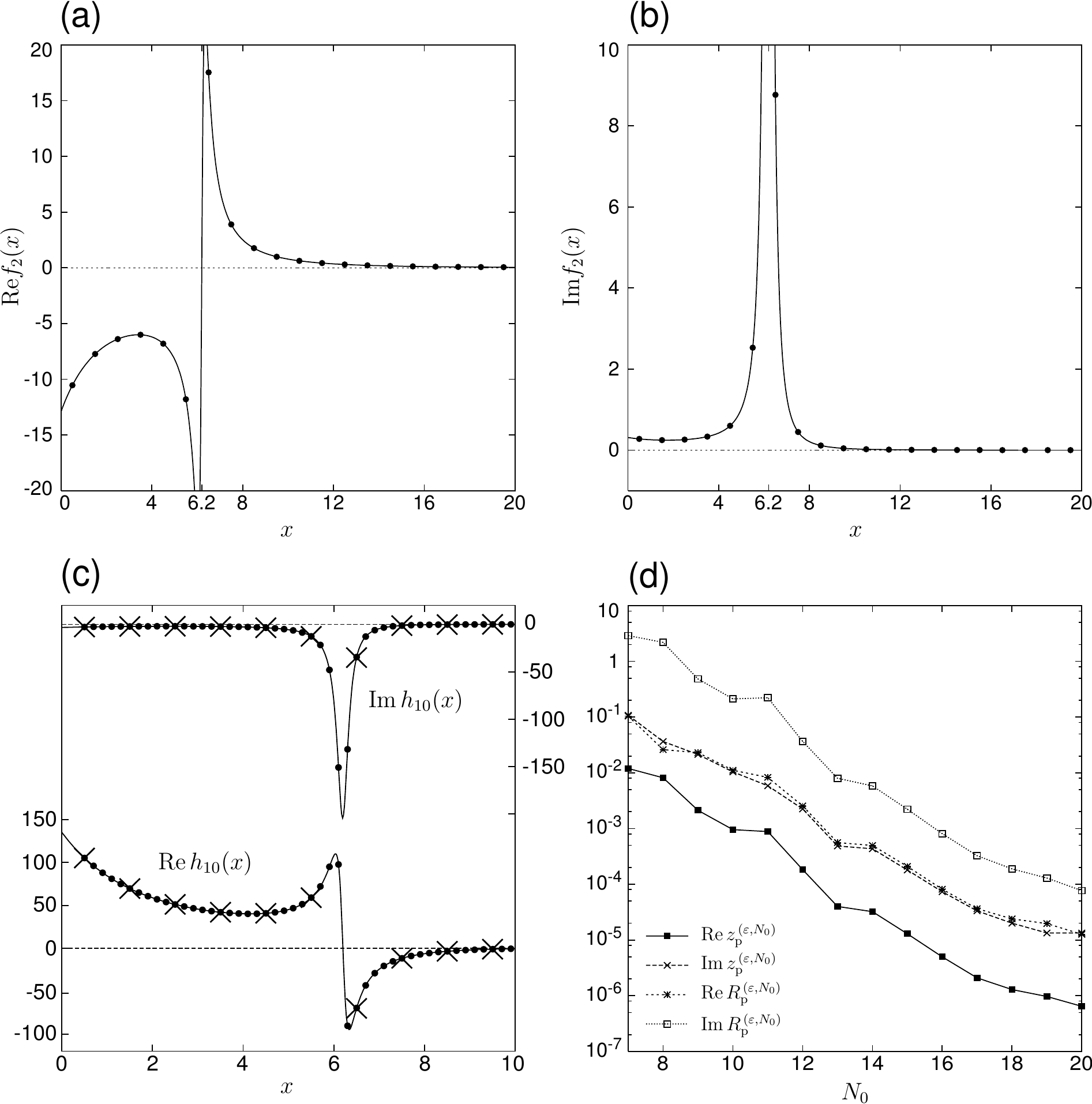}
\caption{\label{fig:5}
\small\textsl{
Analysis of the function $f_2(z)$ with one first order pole in $\Real z>0$ 
(see also the legend of Fig. \ref{fig:4}). $N_0=60$, $\varepsilon\equiv\varepsilon_\mathrm{R}$.
{\bf (a)} and {\bf (b)}: Reconstructed samples $\hf_j^{\,(\rmp;\varepsilon,N_0)}$ (filled dots) 
of $f_2(x)$ (solid line) at the half--integers $x_j$, computed by Eq. \eqref{55.23}.
The truncation number used in \eqref{55.23} is 
$\hm_\rmt^{(\rmp)}(\varepsilon,N_0;k) = 40$, which lies within the \emph{range of convergence}
of $z_\rmp^{(\varepsilon,N_0)}(n)$ shown in Fig. \ref{fig:4}b, and minimizes the relative
root mean squared error: $\delta^{(\varepsilon,N_0)}=1.19\times 10^{-3}$.
{\bf (c)}: Interpolation of the function $h_{10}(x)\doteq(x-\frac{21}{2})f_2(x)$ (solid line) 
associated with $f_2(x)$.
The interpolated values (filled dots) have been computed through formula \eqref{3.1.bis}
and, for better visualization, are shown only for values of $x$ on a uniform grid with step $0.2$.
The crosses represent the samples $h_{10}(N+\frac{1}{2})$ ($N\in[0,N_0]$) used as input data.
{\bf (d)}: Relative error in the evaluation of pole location and residue versus the number $N_0$ of
input function samples $f_N$, $N\in[0,N_0]$.}
}
\end{center}
\end{figure}

The occurrence of this extended plateau (its length is 
$L_\rmp \doteq (n_\mathrm{max}-n_\mathrm{min})= 416$)\footnote{
Actually, we have two lengths $L_\rmp^\rmR$ and $L_\rmp^\rmI$ associated with
the \emph{range of convergence} of the real and of the imaginary part, respectively.
Since they are usually very similar, for simplicity we will frequently refer only 
to $L_\rmp = L_\rmp^\rmR \simeq L_\rmp^\rmI$.}
guarantees that the coefficients $c_n^{\,(\rmp;\varepsilon,N_0)}(k)$ (see \eqref{pullo.1}),
from which the values of $z_\rmp^{(\varepsilon,N_0)}(n)$ follow, and which
are expected to vanish within a certain interval $n_\mathrm{min}<n<n_\mathrm{max}$ 
(see Section \ref{subse:algorithm:finite}), are nearly constant indeed in this interval of $n$--values.
The actual value of this constant, along with the extension of the plateau 
(or \emph{range of convergence}), is determined by the values of $\varepsilon$ and $N_0$.
If $\varepsilon$ is ``sufficiently small'' and $N_0$ is ``sufficiently large''
we have that within this \emph{range of convergence} $c_n^{\,(\rmp;\varepsilon,N_0)}(k)\simeq 0$
(see also \eqref{4.99.tris}), which implies 
$z_\rmp^{(\varepsilon,N_0)}(n) \simeq z_\rmp^{(0,\infty)}(n) \simeq z_\rmp$ 
(see \eqref{L1} and \eqref{L2.bis}).
That this is actually the case, and that the estimate $z_\rmp^{(\varepsilon,N_0)}$
obtained from the set of values of $z_\rmp^{(\varepsilon,N_0)}(n)$ within the \emph{range of convergence}
is indeed close to the true value $z_\rmp$, will be confirmed (or refuted) \emph{a posteriori}
by the study of the sum $\hM_{k}^{\,(\rmp;\varepsilon,N_0)}(m)$ against $m$ (see 
\eqref{L6} and Fig. \ref{fig:4}d), and by evaluating the mean error $\delta^{(\varepsilon,N_0)}$ 
defined in \eqref{num1}.

Patterns of the type shown in Fig. \ref{fig:4}b, in which a \emph{range of convergence} (or a plateau) 
is present and has to be localized, are ubiquitous in our analysis; in practice, the plateau is 
selected as the most extended range of consecutive $n$--values for which all the values of the 
plotted quantity lie within a band, whose width $W_\mathrm{p}$ is a small percentage of its 
central value (with only numerical roundoff error, i.e, 
$\varepsilon=\varepsilon_\mathrm{R}$, the bin--width we used is $W_\mathrm{p}=10^{-3}\,\%$ 
of the central value of the band); for instance, referring to Fig. \ref{fig:4}b, 
we have that $\Real z_\rmp^{(\varepsilon,N_0)}(n)$ lies within the range $[6.199969:6.200031]$
for $n\in [10:426]$, i.e., the length of the \emph{range of convergence} is $L_\mathrm{p}=416$ 
(only values up to $n=100$ are shown). Once the \emph{range of convergence} has been located, 
the value of the quantity being considered is chosen to be the sample mean of the values 
belonging to the \emph{range of convergence}, while the sample standard deviation is used 
as an estimate of the uncertainty; in the example of Fig. \ref{fig:4}b we obtain 
$\Real z_\rmp^{(\varepsilon,N_0)} = 6.20000005 \pm 1.5 \times 10^{-7}$,
which is in excellent agreement with the true value $\Real z_\rmp = 6.2$.

A situation very similar to that just discussed for the position of the pole
arises from the analysis of the residue $R_\rmp$ of $f_2(z)$ in $z_\rmp$. 
After that the location $z_\rmp^{(\varepsilon,N_0)}$ of the pole has been computed, use can be made of 
formula \eqref{L4} to obtain $R_\rmp^{(\varepsilon,N_0)}(n)$
(see Fig. \ref{fig:4}c). Even in this case, after some wild oscillations for small values of $n$,
the value of $R_\rmp^{(\varepsilon,N_0)}(n)$ stabilizes on a value $R_\rmp^{(\varepsilon,N_0)}$, which is 
very close to the true value $R_\rmp$ of the residue (see the figure legend for numerical details).

We can now proceed to test the reliability of our results.
The estimates just computed of pole location $z_\rmp^{(\varepsilon,N_0)}$ and residue 
$R_\rmp^{(\varepsilon,N_0)}$ can be inserted into the terms $\hc_{n,k}^{\,(\rmp;\varepsilon,N_0)}$ 
(see \eqref{4.3.bis}) in order to compute the sum $\hM_k^{\,(\rmp;\varepsilon,N_0)}(m)$ given by
formula \eqref{L6}. The plot of the latter, as a function of $m$, is shown in Fig. \ref{fig:4}d 
for various values of $k$; the clear presence of the plateaux validates (qualitatively) the 
correctness of the values of $z_\rmp^{(\varepsilon,N_0)}$ and $R_\rmp^{(\varepsilon,N_0)}$ 
(cf. Fig. \ref{fig:4}a). It is worth observing how the roundoff noise
and the finiteness of $N_0$ limit superiorly the length of the plateaux,
causing, for $m>\hm^{(\rmp)}(\varepsilon,N_0;k)\simeq 230$, 
the divergence of $\hM_k^{\,(\rmp;\varepsilon,N_0)}(m)$ as a power of $m$ (see \eqref{55.9}).

The values of $z_\rmp^{(\varepsilon,N_0)}$ and $R_\rmp^{(\varepsilon,N_0)}$ can be further 
validated by comparing the input samples $(f_2)_N^{(\varepsilon)}$ with the approximate
samples $\hf_k^{\,(\rmp;\varepsilon,N_0)}$, which can be calculated by means of \eqref{55.23}.
These reconstructed samples $\hf_k^{\,(\rmp;\varepsilon,N_0)}$, shown in 
Figs. \ref{fig:5}a and \ref{fig:5}b (filled dots), reproduce extremely well the input 
data $(f_2)_N^{(\varepsilon_\mathrm{R})}$, and the small value of the root mean squared error 
$\delta^{(\varepsilon,N_0)}=1.19\times 10^{-3}$ (see Eq. \eqref{num1}) confirms the 
great accuracy of the values of pole location and residue just recovered.

\begin{figure}[htb]  
\captionsetup{width=0.95\textwidth}
\begin{center}
\leavevmode
\includegraphics[width=11.4cm]{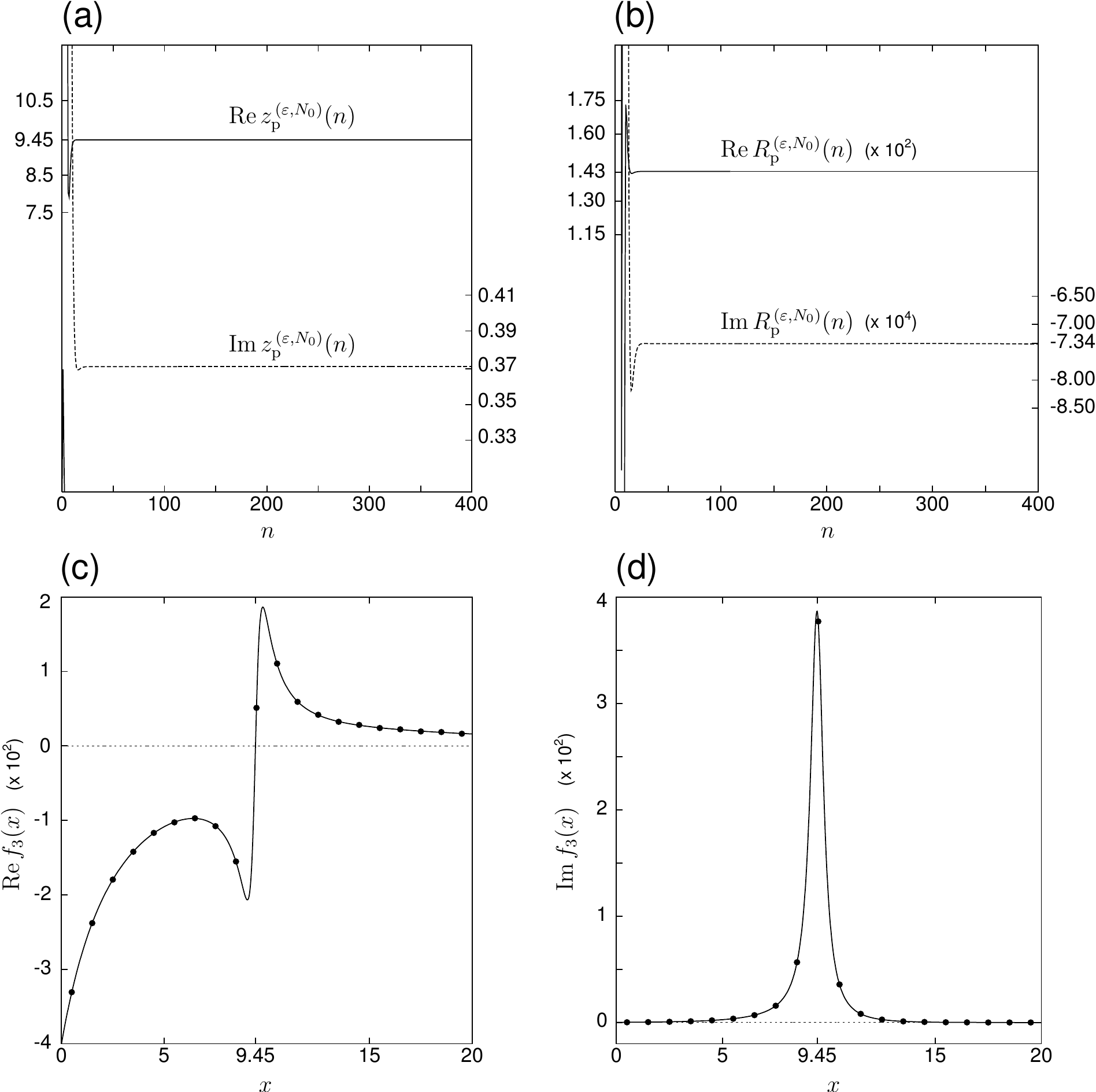}
\caption{\label{fig:6}
\small\textsl{Analysis of a function with one first order pole in $\Real z>0$. $N_0=60$; 
$\varepsilon\equiv\varepsilon_\mathrm{R}$.
$\sds f_3(z)=[\cosh(z-z_\rmp)-1]\{(z+5)^2[\sinh(z-z_\rmp)-(z-z_\rmp)]\}^{-1}$.
$z_\rmp=9.45+0.37\rmi$, $R_\rmp=1.433941\times 10^{-2}-7.348186\times 10^{-4}\rmi$.
{\bf (a)}: Recovery of pole location. $z_\rmp^{(\varepsilon,N_0)}(n)$ vs. $n$ (see \eqref{Q5.2}).
Length of the \emph{ranges of convergence} with $W_\mathrm{p} = 10^{-3}\%$: 
$L_\mathrm{p}^\mathrm{R} = 552$, $L_\mathrm{p}^\mathrm{I} = 489$.
The computed pole position is:
$z_\rmp^{(\varepsilon,N_0)} = (9.4500018 \pm 4.7 \times 10^{-6})+(3.700014 \times 10^{-1} 
\pm 4.8 \times 10^{-6})\rmi$.
{\bf (b)}: Recovery of pole residue. Upper part of the panel: $\Real R_\rmp^{(\varepsilon,N_0)}(n)$ 
(multiplied by $10^2$) vs. $n$ (see \eqref{L4}).
Lower part of the panel: $\Imag R_\rmp^{(\varepsilon,N_0)}(n)$ (multiplied by $10^4$) vs. $n$.
Length of the \emph{ranges of convergence} with $W_\mathrm{p} = 10^{-3}\%$: 
$L_\mathrm{p}^\mathrm{R} = 553$, $L_\mathrm{p}^\mathrm{I} = 524$.
The computed residue is:
$R_\rmp^{(\varepsilon,N_0)} = (1.4339431 \times 10^{-2} \pm 1.4 \times 
10^{-8})-(7.34952 \times 10^{-4} \pm 1.2 \times 10^{-8})\rmi$.
{\bf (c)} and {\bf (d)}: Real and imaginary parts of the reconstructed
samples $\hf_j^{\,(\rmp;\varepsilon,N_0)}$ (filled dots) of $f_3(x)$ (solid line) 
at the half--integers $x_j$, computed by Eq. \eqref{55.23}:
$\hm_\rmt^{(\rmp)}(\varepsilon,N_0;k)=36$ for every $k$; 
$\delta^{(\varepsilon,N_0)} = 1.17\times 10^{-2}$.}
}
\end{center}
\end{figure}

Now that position and residue of the pole are known, we can use the interpolation 
formula \eqref{3.1.bis} (see also the related formula \eqref{3.1}, which holds in 
the case of an infinite number of noiseless input data) to obtain the values of the 
meromorphic function for every $x>0$. An example of such a calculation is given in
Fig. \ref{fig:5}c, where the (approximate) values of the meromorphic function 
$h_{10}(x) \doteq (x-\frac{21}{2})f_2(x)$ (which has a first order pole at 
$z=z_\rmp$ with residue $(z_\rmp-\frac{21}{2})R_\rmp$), have been computed by means of Eq. \eqref{3.1.bis}
(indicated by the filled dots only on a regular grid for better visualization) and 
compared with the true $h_{10}(x)$ (solid line).

\begin{figure}[tb]  
\captionsetup{width=0.95\textwidth}
\begin{center}
\leavevmode
\includegraphics[width=11.4cm]{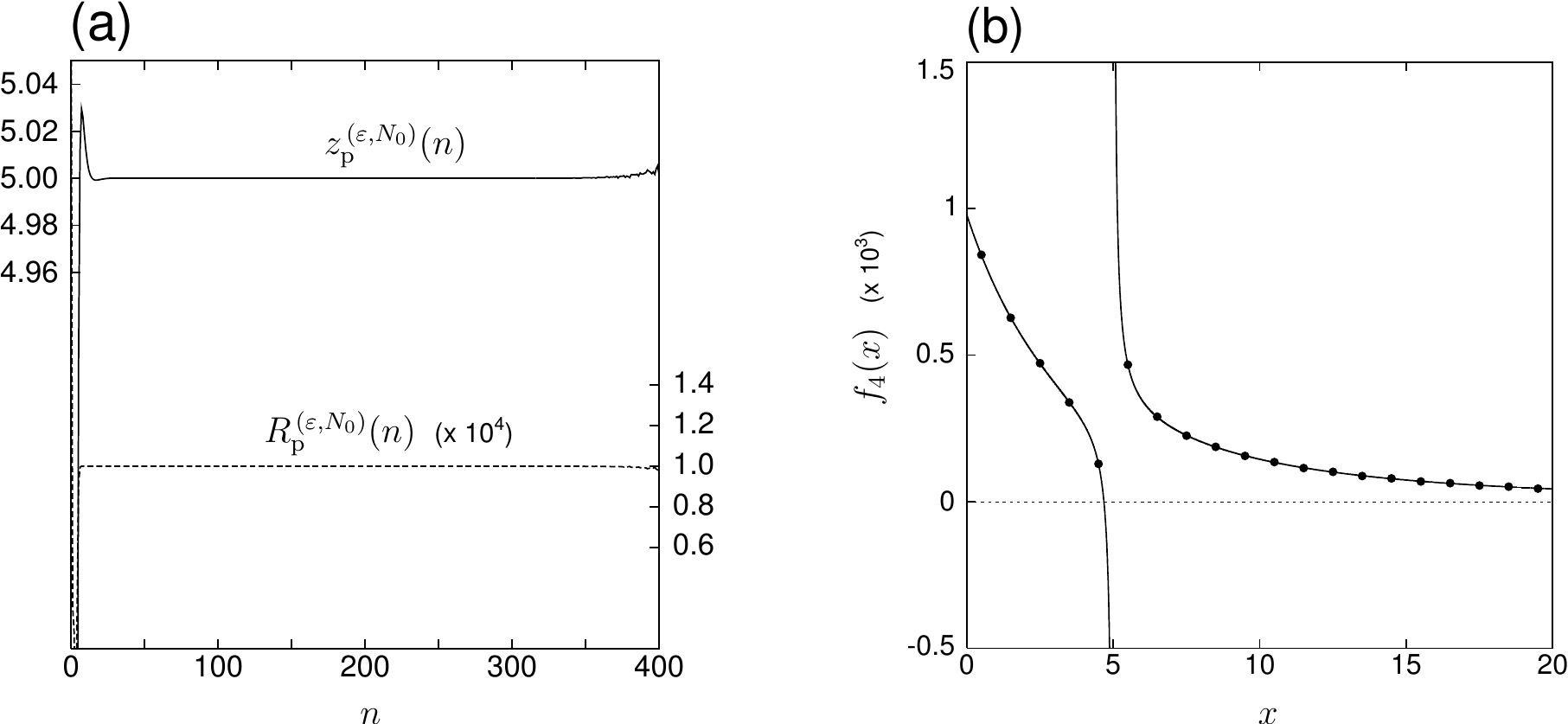}
\caption{\label{fig:7}
\small\textsl{Analysis of a function with a first order pole close to a zero.
$\sds f_4(z)=(z+10)^{-3}+\eta (z-z_\rmp)^{-1}$.
$\eta=10^{-4}$; $z_\rmp=5.0$, $R_\rmp=10^{-4}$; $N_0=60$, 
$\varepsilon\equiv\varepsilon_\mathrm{R}$. $W_\mathrm{p} = 10^{-2}\%$.
The position of the zero is $z_0=4.68343$.
{\bf (a)}: Upper part of the panel: $z_\rmp^{(\varepsilon,N_0)}(n)$ vs. $n$; $L_\mathrm{p} = 318$.
Computed pole position: $z_\rmp^{(\varepsilon,N_0)} = 5.000003 \pm 2.1 \times 10^{-5}$.
Lower part of the panel: $R_\rmp^{(\varepsilon,N_0)}(n)$ (multiplied by $10^4$) vs. $n$; $L_\mathrm{p} = 295$.
Computed residue: $R_\rmp^{(\varepsilon,N_0)} = 9.999946 \times 10^{-5} \pm 6.4 \times 10^{-10}$.
{\bf (b)}: Reconstructed samples $\hf_j^{\,(\rmp;\varepsilon,N_0)}$ (filled dots) of $f_4(x)$ (solid line) 
at the half--integers $x_j$, computed by Eq. \eqref{55.23}:
$\hm_\rmt^{(\rmp)}(\varepsilon,N_0;k)=82$ for every $k$; $\delta^{(\varepsilon,N_0)} = 9.91\times 10^{-3}$.}
}
\end{center}
\end{figure}

Finally, Fig. \ref{fig:5}d illustrates the accuracy of the computation of pole location 
and residue as a function of the number of input data $N_0$. It can be seen that the 
relative error on $z_\rmp^{(\varepsilon,N_0)}$ and $R_\rmp^{(\varepsilon,N_0)}$ rapidly 
decreases when the number of input samples increases, being quite small even in the case 
of a rather limited number of (accurate) input data.

The example discussed so far is typical of an extensive numerical experimentation performed 
with numerous test functions, which, for brevity, are not reported here. The analysis of an 
other example is shown in Fig. \ref{fig:6}, where we study the function 
$f_3(z)=\frac{1}{(z+5)^2}\frac{\cosh(z-z_\rmp)-1}{\sinh(z-z_\rmp)-(z-z_\rmp)}$, which features
a \emph{non--trivial} first order pole at $z=z_\rmp$, which is what is left by the cancellation 
of a second order zero in the numerator with a third order zero in the denominator.
The analysis of this function follows faithfully the one made previously for the function $f_2(z)$;
it is worth remarking in the plots of $z_\rmp^{(\varepsilon,N_0)}(n)$ and $R_\rmp^{(\varepsilon,N_0)}(n)$ 
versus $n$ (see Figs. \ref{fig:6}a,b) the large extent of the interval in which these functions 
are practically equal to their corresponding true values.

A well-known ``defect'' affecting the methods for pole recovery which are based on rational approximants,
notably the Pad\'e approximant method \cite{Baker2}, is the difficulty in handling effectively situations 
when a pole is very close to a zero of the function.
A typical example of such a situation is the case of a function of the type \cite[p. 57]{Baker}:
$f(z)=g(z)+\eta(z-z_\rmp)^{-1}$, where $g(z)$ is analytic: if the real parameter $\eta$ is small,
the pole $z_\rmp$ can get close to a zero $z_0$ of $f(z)$.
Instead, the method we propose is, as expected, practically insensitive to this kind of problem. 
To exemplify our analysis we have taken $g(z)=(z+10)^{-3}$, which is analytic in $\Real z>0$, 
and $z_\rmp=5.0$; the results are summarized in Fig. \ref{fig:7} and Table \ref{tab:1}.

\begin{table}[tb]
\captionsetup{width=0.95\textwidth}
\scriptsize
\caption{\label{tab:1}
\small\textsl{Analysis of a function with a first order pole close to a zero.
$f_4(z)=(z+10)^{-3}+\eta(z-z_\rmp)^{-1}$. $z_\rmp=5.0$; $N_0=60$, $\varepsilon=0$.
The second column gives the location of the real zero of $f_4(z)$ in $\Real z>0$, 
which is close to $z_\rmp$. The fourth column gives the range of $n$ values within 
which $z_\rmp^{(\varepsilon,N_0)}(n)$, computed by Eq. \eqref{Q5.2},
differs from the true pole position by less than $W_\mathrm{p}=0.01\,\%$: i.e., it belongs to
$[4.99975:5.00025]$ (see also Fig. \ref{fig:7}a).}
}
\begin{center}
\leavevmode
\begin{tabular}{ccccc}
\\ \hline \\
$\eta$ & Zero & Zero--pole & Range of $n$ values where & Plateau length $L_\mathrm{p}$ \\
 & location & distance & $z_\rmp^{(\varepsilon,N_0)}(n) \in [4.99975:5.00025]$ & with $W_\mathrm{p}=10^{-2}\,\%$
\\[+2pt] \hline \\
$10^{-4}$ & $4.68343$ & $0.31657$& $n\in[23:364]$ & 342 \\[2pt]
$10^{-5}$ & $4.96637$ & $0.03363$& $n\in[36:327]$ & 292 \\[2pt]
$10^{-6}$ & $4.99662$ & $0.00338$& $n\in[55:296]$ & 242 \\[2pt]
$10^{-7}$ & $4.99966$ & $0.00034$& $n\in[93:257]$ & 165 \\[2pt]
\hline
\end{tabular}
\end{center}
\end{table}

In Fig. \ref{fig:7}a we see, for $\eta=10^{-4}$, that both pole position and residue are 
correctly identified, and that the range over which $z_\rmp^{(\varepsilon,N_0)}(n)$ and 
$R_\rmp^{(\varepsilon,N_0)}(n)$ are nearly constant is rather extended 
($z_\rmp^{(\varepsilon,N_0)}(n)$ is within a band centered around the true value, with 
$W_\mathrm{p} = 10^{-2}\,\%$, for $23 \leqslant n \leqslant 350$).
The accuracy of the values of $z_\rmp^{(\varepsilon,N_0)}$ and $R_\rmp^{(\varepsilon,N_0)}$
obtained from the data shown in Fig. \ref{fig:7}a, is then verified in Fig. \ref{fig:7}b, 
which exhibits the excellent reconstruction of the data samples $\hf_k^{\,(\rmp;\varepsilon,N_0)}$
($\delta^{(\varepsilon,N_0)}\sim 10^{-2}$) in spite of the great closeness between 
pole and zero of $f_4(x)$.
Table \ref{tab:1} shows that these results hold even for much smaller values of $\eta$. 
For instance, with $\eta=10^{-7}$ the zero--pole distance becomes tiny 
($\sim 3.4\times 10^{-4}$) but, nevertheless, the location of the pole 
is recovered correctly (within a $0.01\%$ error) over a very extended range 
of $n$--values ($n_\mathrm{min}(\varepsilon,N_0;k)=74$ and $n_\mathrm{max}(\varepsilon,N_0;k)=274$).

Finally, in Fig. \ref{fig:8} we summarize the analysis in the case of noisy 
input samples $f_N^{(\varepsilon)}$. The input data have been obtained from
the noiseless samples $f_N$ by adding white noise uniformly distributed in the 
interval $[-\varepsilon,\varepsilon]$, in such a way that for all $N$ we have:
$|(f_N-f_N^{(\varepsilon)})/f_N|\leqslant\varepsilon$. For this analysis we used 
the test function: $f_5(z)=\frac{1}{(z+5)^5}\frac{1}{(z-z_\rmp)}$.
Figure \ref{fig:8}a shows $\Real z_\rmp^{(\varepsilon,N_0)}(n)$ as a function of $n$
(see Eq. \eqref{Q5.2}), computed for various values of the noise bound $\varepsilon$.
These plots show how the range of $n$--values over which the value of 
$\Real z_\rmp^{(\varepsilon,N_0)}(n)$
remains constant (and almost equal to the true value $\Real z_\rmp=5.2$) gets 
shorter as the level of noise increases. When $\varepsilon$ is varied we have that
$n_\mathrm{min}(\varepsilon,N_0;k)$ remains nearly constant, whereas (as expected)
$n_\mathrm{max}(\varepsilon,N_0;k)$ changes quite considerably
(e.g., $n_\mathrm{max}(10^{-4},N_0;k)=65$ and $n_\mathrm{max}(10^{-2},N_0;k)=20$).
This behavior exemplifies the discussion of formula \eqref{L2.bis} made in Section
\ref{subse:algorithm:finite}:
the limits in \eqref{L2.bis} cannot be interchanged and, as a consequence,
when $\varepsilon>0$ and $N_0<\infty$ are kept fixed, $z_\rmp^{(\varepsilon,N_0)}(n)$ 
diverges as $n\to\infty$. The value of $n$ where such a divergence ``approximately''
sets in, which we denoted by $n_\mathrm{max}$, depends on $\varepsilon$ and $N_0$:
it is finite when $\varepsilon>0$ and $N_0<\infty$, and it is such that
$\lim_{\substack{{N_0\rightarrow+\infty}\\{\varepsilon\rightarrow 0}}}
n_\mathrm{max}(\varepsilon,N_0) = +\infty$.

\begin{figure}[htb]  
\captionsetup{width=0.95\textwidth}
\begin{center}
\leavevmode
\includegraphics[width=11.4cm]{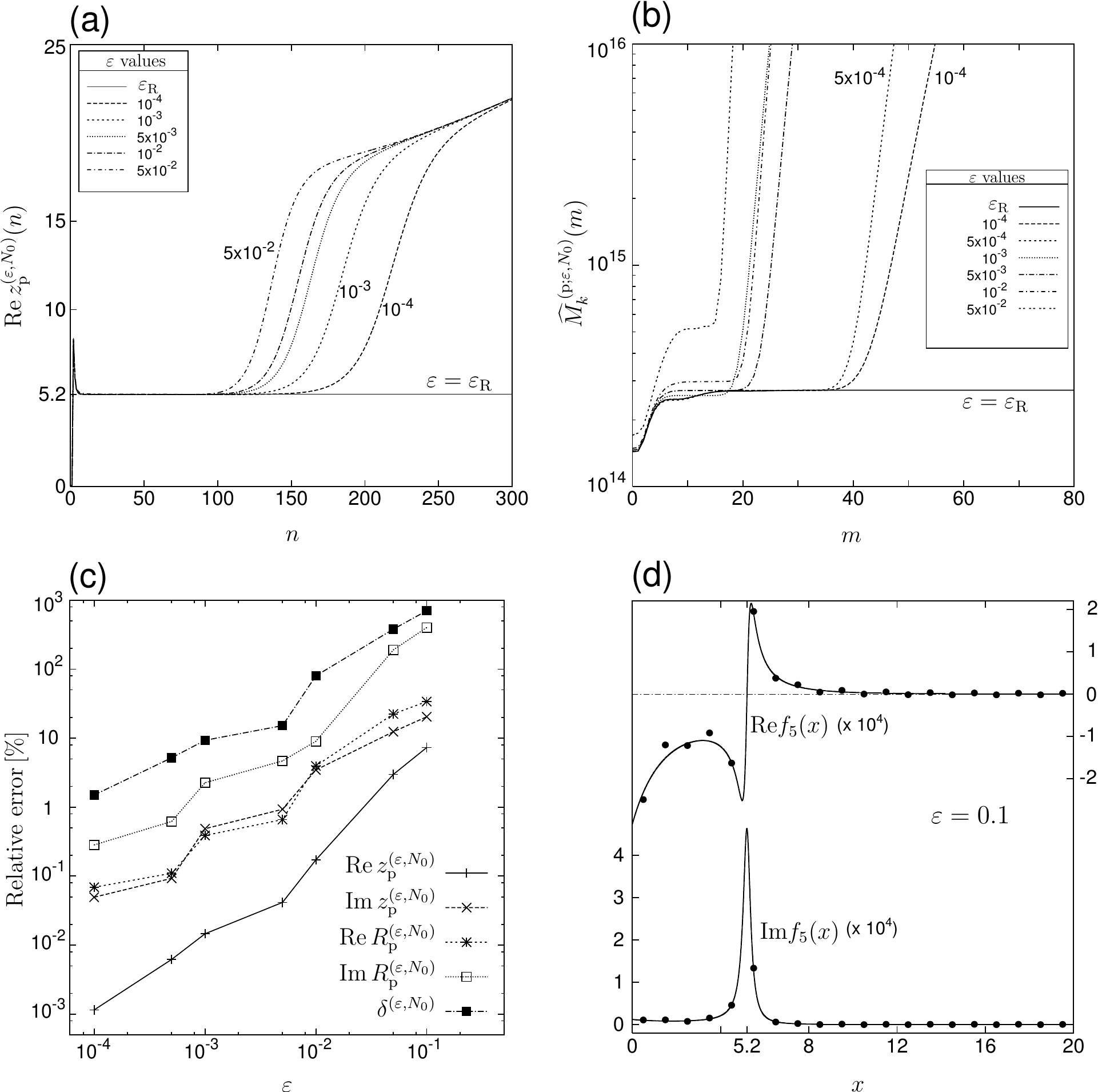}
\caption{\label{fig:8} 
\footnotesize\textsl{Analysis of a function with one first order pole
in $\Real z>0$ with noisy input samples $f_N^{(\varepsilon)}$:
$f_5(z)=(z+5)^{-5}(z-z_\rmp)^{-1}$. $z_\rmp=5.2+0.2\rmi$,
$R_\rmp=9.005168\times 10^{-6}-8.855849\times 10^{-7}\rmi$; $N_0=60$.
The noisy function samples $f_N^{(\varepsilon)}$ have been computed
as $f_N^{(\varepsilon)}=(1+\nu_N^{(\varepsilon)})f_N$, $\nu_N^{(\varepsilon)}$
being random variables uniformly distributed in the interval
$[-\varepsilon,\varepsilon]$, and $f_N$ are the noiseless function
samples.
{\bf (a)}: Function $\Real z_\rmp^{(\varepsilon,N_0)}(n)$ vs. $n$, computed
by Eq. \eqref{Q5.2} with various values of noise bound $\varepsilon$.
Bin--width $W_\mathrm{p}=10^{-2}\%$.
Length of the \emph{range of convergence} and computed value of $\Real z_\rmp^{(\varepsilon,N_0)}$ are:
$\varepsilon=\varepsilon_\mathrm{R}$ (i.e., only numerical roundoff error): 
$L_\mathrm{p}^\mathrm{R} = 399$, $\Real z_\rmp^{(\varepsilon,N_0)} = 5.199999 \pm 1.7 \times 10^{-5}$; 
$\varepsilon=10^{-4}$: $L_\mathrm{p}^\mathrm{R} = 70$, 
$\Real z_\rmp^{(\varepsilon,N_0)} = 5.20006 \pm 1.4 \times 10^{-4}$; 
$\varepsilon=10^{-3}$: $L_\mathrm{p}^\mathrm{R} = 41$, 
$\Real z_\rmp^{(\varepsilon,N_0)} = 5.19984 \pm 1.7 \times 10^{-4}$;
$\varepsilon=5 \times 10^{-3}$: $L_\mathrm{p}^\mathrm{R} = 27$,
$\Real z_\rmp^{(\varepsilon,N_0)} = 5.20168 \pm 6.3 \times 10^{-4}$;
$\varepsilon=10^{-2}$: $L_\mathrm{p}^\mathrm{R} = 24$, 
$\Real z_\rmp^{(\varepsilon,N_0)} = 5.19519 \pm 9.7 \times 10^{-4}$; 
$\varepsilon=5 \times 10^{-2}$: $L_\mathrm{p}^\mathrm{R} = 15$, 
$\Real z_\rmp^{(\varepsilon,N_0)} = 5.1951 \pm 5.4 \times 10^{-3}$.
{\bf (b)}: Plot of $\hM_{k}^{\,(\rmp;\varepsilon,N_0)}(m)$ vs. $m$ (see \eqref{L6}),
computed with different values of $\varepsilon$; $k=13$. 
{\bf (c)}: Relative error in the evaluation of pole location and residue versus the noise bound
$\varepsilon$.
{\bf (d)}: Reconstructed samples $\hf_j^{\,(\rmp;\varepsilon,N_0)}$ (filled dots) of $f_5(x)$ 
(multiplied by $10^4$, and plotted with solid line) at the half--integers $x_j$, computed by Eq. \eqref{55.23}; 
$\varepsilon=10^{-1}$. $\hm_\rmt^{(\rmp)}(\varepsilon,N_0;k)=11$ for every $k$, 
$\delta^{(\varepsilon,N_0)} = 1.105$.}
}
\end{center}
\end{figure}

However, even in the presence of quite noisy input data, this interval remains 
extended enough to allow a reliable recovery of $\Real z_\rmp$ (see the figure legend for numerical details). 
Correspondingly, in panel (b), the sum $\hM_k^{\,(\rmp;\varepsilon,N_0)}(m)$, 
computed with pole location $z_\rmp^{(\varepsilon,N_0)}$ and residue $R_\rmp^{(\varepsilon,N_0)}$
obtained from the analysis of the data in (a), displays the same behavior, 
that is, shorter plateaux for higher levels of noise.
The accuracy achieved in the computation of the pole parameters is given in Fig. \ref{fig:8}c. 
We see that the relative error remains always quite satisfactory: for instance, the real part 
of the pole $\Real z_\rmp$ is computed within nearly $1\%$ when the input data suffer of 
(at most) a $10\%$ error. Also $R_\rmp$ exhibits a similar behavior, though its
estimate always results less accurate than that of $z_\rmp$. 
Finally, in Fig. \ref{fig:8}d an example of reconstruction
of the data samples $\hf_k^{\,(\rmp;\varepsilon,N_0)}$ (with $\varepsilon=0.1$) is shown.

\section{Concluding remarks and extensions}
\label{se:conclusions}

The method we propose is able to compute estimates of location and residue of a single first
order pole of a function meromorphic in $\Real z > 0$ from a finite set of noisy samples taken on
a uniform grid of points spaced one unit apart on the real positive semi--axis.
Moreover, the degree of approximation of these estimates to the true values $z_\rmp$ and $R_\rmp$
can be evaluated by computing the relative mean squared error $\delta^{(\varepsilon,N_0)}$. \\
In conclusion, the following comments and remarks are in order.

(1) A limit of the method we have proposed consists in the fact that the pole parameters 
cannot be determined if the pole is located outside the range of the data set: e.g., 
when the pole lies in the half--plane $\Real z<0$, while the input data 
$\{f_N^{(\varepsilon)}\}_{N=0}^{N_0}$ are given in the half--plane $\Real z>0$,
or if we have a pole at $z_\rmp$ while the data are given only up to 
$z=N_0+\frac{1}{2}<\Real z_\rmp$. In other words, the singularities
should lie well inside a region where the function is partially known through the 
data set $\{f_N^{(\varepsilon)}\}_{N=0}^{N_0}$.

(2) As a typical example of a physical problem which can be properly tackled by our method, 
it can be considered the following one: suppose that, for various values of the angular 
momentum $\ell$, a finite set of partial--waves $a_\ell$, at fixed energy, has been determined 
in a scattering process: i.e., the data set is given by $\{a_\ell\}_{\ell=0}^{L_0}$. By means 
of our method, we can explore whether these partial--waves are the restriction to the integers 
of a function which is analytic in a certain domain: e.g., in the half--plane $\Real\lambda>-\frac{1}{2}$
($\lambda\in\C; \lambda=\ell+\rmi\nu$). If, instead, some resonances are present in the 
collision process, location and residue of the pole which represents these resonances can
be determined. This analysis is particularly relevant in the inverse scattering problem 
at fixed energy, especially in the case of Yukawian potentials, whose partial--wave amplitudes 
are known to satisfy Carlson's bound.

(3) The sampling rate of the input data can be generalized taking as input data the set
$\{f_N^{(\alpha)}\}_{N=0}^\infty$ made of samples $f_N^{(\alpha)} \doteq f[\alpha(N+\ud)]$
($N\in\N$, $\alpha>0$) of a meromorphic function $f(z)$ taken at equidistant
real points, $\alpha$ units apart.

(4) Functions $f(z)$ with a pole of any order higher than unity can be analyzed, 
and an algorithmic procedure capable to return the pole parameters, i.e., 
location and Laurent coefficients, can also be given.
The case of function $f(z)$ with more than one pole in $\Real z>0$ can also be considered, and
even in this case an algorithmic procedure for recovering location and residue
of each pole can be presented. These latter extensions will be the argument 
of a forthcoming paper.

\newpage

\setcounter{equation}{0}
\renewcommand{\theequation}{A.\arabic{equation}}
\setcounter{section}{0}
\renewcommand{\thesection}{A}

\section*{Appendix}
\label{appendix}

\textbf{(A)} For the reader's convenience, some properties of the 
Pollaczek polynomials $P_n(y)$ are here briefly summarized \cite{Bateman}.
\\[+6pt]
The definition of $P_n(y)$ in terms of Gauss hypergeometric series reads:
\beq
P_n(y) \doteq P_n^{(1/2)}(y) = \rmi^n \, 
\null_2F_1\left(-n,\ud+\rmi y;1;2\right).
\label{AA0}
\eeq
The polynomials $P_n(y)$ satisfy the following recurrence relation:
\beq
\begin{split}
& (n+1)P_{n+1}(y)-2yP_n(y)+nP_{n-1}(y)=0, \\
& P_{-1}(y)=0, \quad P_0(y)=1.
\end{split}
\label{AA1}
\eeq
The polynomials $P_n(y)$ are orthonormal with respect to the weight function 
$w(y)=\frac{1}{\pi}|\Gamma(\ud+\rmi y)|^2=(\cosh\pi y)^{-1}$, i.e.:
\beq
\int_{-\infty}^{+\infty}P_m(y)\,P_n(y)\,w(y)\,\rmd y = \delta_{m,n} \qquad (m,n\in\N).
\label{AA2}
\eeq

\skip 0.5cm

\noindent 
\textbf{(B)} In Ref. \cite{DeMicheli1} (see formula (75)) we have proved the following asymptotic 
formula for $P_n(z)$, for large values of $n$ (at fixed $z\in\C$):
\beq
\label{B1}
P_n(z) \staccrel{\sds\sim}{n \gg 1}
{\rmi^n}\left[\frac{(2n)^{-\rmi z-1/2}}{\Gamma(\frac{1}{2}-\rmi z)}+ (-1)^n 
\frac{(2n)^{\rmi z-1/2}}{\Gamma(\frac{1}{2}+\rmi z)}
\right]
\qquad (z \in\C~\mathrm{fixed}).
\eeq
If $\Real z>0$, it follows that
\beq
\label{B3}
P_n(-\rmi z)
\staccrel{\sds\sim}{n \gg 1}
\frac{(-\rmi)^n}{\Gamma\left(\frac{1}{2}+z\right)}(2n)^{z-1/2}
\qquad (z \in\C~\mathrm{fixed},\Real z>0).
\eeq
If in formula \eqref{B3} we put $z=(N+\frac{1}{2})$ ($N\in\N$), we have:
\beq
\label{B33}
P_n\left[-\rmi\left(N+\frac{1}{2}\right)\right]
\staccrel{\sds\sim}{n \gg 1}
\frac{(-\rmi)^n}{N!}\, (2n)^{N}
\qquad (N\in\N).
\eeq

\end{document}